\documentclass[]{interact}

\usepackage{epstopdf}
\usepackage[caption=false]{subfig}

\usepackage[]{todonotes}

\usepackage{tikz}
\usetikzlibrary{calc}
\usetikzlibrary{arrows,arrows.meta,intersections,shapes.geometric,positioning,calc,patterns}

\usepackage{paralist}
\usepackage[shortlabels]{enumitem}

\usepackage{hyperref}
\usepackage{algorithmic}
\usepackage[ruled,linesnumbered,vlined]{algorithm2e}

\usepackage{resizegather}

\usepackage[numbers,sort&compress]{natbib}
\bibpunct[, ]{[}{]}{,}{n}{,}{,}
\renewcommand\bibfont{\fontsize{10}{12}\selectfont}

\theoremstyle{plain}
\newtheorem{theorem}{Theorem}[section]
\newtheorem{lemma}[theorem]{Lemma}

\newtheorem{proposition}[theorem]{Proposition}
\newtheorem{assumption}{Assumption}

\theoremstyle{definition}
\newtheorem{definition}[theorem]{Definition}
\newtheorem{example}[theorem]{Example}

\theoremstyle{remark}
\newtheorem{remark}{Remark}

\DeclareMathOperator*{\argmin}{\arg\min}

\begin{document}

\articletype{ARTICLE}

\title{A heavy-ball type curve search method for smooth convexly constrained optimization}

\author{
\name{Federica Donnini\textsuperscript{a} and Pierluigi Mansueto\textsuperscript{a}\thanks{CONTACT Pierluigi Mansueto. Email: pierluigi.mansueto@unifi.it}}
\affil{\textsuperscript{a} Department of Information Engineering, University of Florence, Via di Santa Marta 3, 50139, Florence, Italy}
}

\maketitle

\begin{abstract}
This paper addresses smooth convexly constrained optimization problems where the Euclidean projection onto the feasible set is computationally tractable. Although momentum techniques like Polyak's heavy-ball are known for accelerating optimization algorithms, their use in constrained settings remains limited due to challenges in preserving feasibility and ensuring convergence. We thus propose a heavy-ball-type method that extends to the constrained case a recently introduced curve-search globalization framework. The method attempts a momentum update and performs a curvilinear search to enforce an Armijo-type descent condition: when the momentum step is infeasible or unacceptable, the algorithm smoothly reverts to a feasible descent direction. We prove that the algorithm is well-defined and globally convergent to stationary points; the derivation of these results is nontrivial due to the use of a heavy-ball type direction in a constrained setting, where it may generate infeasible iterates. We discuss the incorporation of further mechanisms into the algorithm, including non-monotone curve search, spectral steplength selection and an adaptive momentum strategy. Numerical experiments on benchmark problems show the method is robust and competitive with the state-of-the-art.
\end{abstract}


\begin{amscode}
	90C26; 90C30
\end{amscode}

\section{Introduction}
\label{sec::introduction}

In this paper, we focus on solving convexly constrained problems of the form
\begin{equation}
    \label{eq::con-min}
    \begin{aligned}
        \min \ & f(x)\\
        \text{s.t. }& g(x)\leq\mathbf{0}_m,
    \end{aligned}
\end{equation}
 where $f:\mathbb{R}^n\to\mathbb{R}$ is a continuously differentiable function and $g:\mathbb{R}^n\to\mathbb{R}^m$ is a vector-valued continuously differentiable convex function defining the feasible set $\Omega = \{x \in \mathbb{R}^n \mid g(x) \le \mathbf{0}_m\}$, with $\mathbf{0}_m$ being a vector of all zeros of size $m$. In the following, we use the notation $[m] = \{1,\ldots, m\}$. Given a subset of constraints indices $I \subseteq [m]$, we define $\Omega_I = \{x \in \mathbb{R}^n\mid g_i(x) \le 0,\ \forall i \in I\}$ to denote the convex feasible set associated with $I$; clearly, $\Omega \subseteq \Omega_{I}$. Finally, we denote by $\Pi_{\Omega}$ the Euclidean projection onto the convex set $\Omega$, which we assume to be computationally tractable.

Classical approaches for solving problems of the form \eqref{eq::con-min} are well established (see, e.g., \cite{bertsekas1999nonlinear,grippo2023introduction}). Starting from a feasible point, methods such as the \textit{projected gradient method} and the \textit{Frank--Wolfe method} iteratively compute feasible descent directions and perform a line search to determine an appropriate steplength. These directions are obtained either by applying the projection operator onto $\Omega$ or by solving an auxiliary smooth  linear optimization subproblem, thereby avoiding explicit projections.
In particular, the work in \cite{birgin2000nonmonotone} introduces \texttt{SPG}, a projected gradient method in which the classical scheme is integrated by a spectral gradient choice of the direction to be projected, and a non-monotone line search based on quadratic interpolation is introduced.
These mechanisms significantly enhance the practical effectiveness and efficiency of the projected gradient approach, making \texttt{SPG} one of the state-of-the-art methods in the considered convexly constrained scenario.

In this paper, we focus on employing accelerated methodologies within a convexly constrained setting. In particular, we consider the well-known Polyak's heavy-ball method \cite{polyak1964some, polyak1987introduction}, whose standard update rule is given by
\begin{equation*}
x_{k+1} = x_k - \alpha \nabla f(x_k) + \beta (x_k - x_{k-1}),
\end{equation*}
where $\alpha, \beta \in \mathbb{R}$ are typically fixed positive parameters. Incorporating information from previous steps when computing the current update helps to mitigate aggressive oscillations in regions of high curvature while providing acceleration in regions of low curvature. This characteristic has made Polyak’s heavy-ball method a milestone in nonlinear unconstrained optimization, especially in stochastic finite-sum problems common in machine learning applications \cite{bottou18, wright2022optimization}.

While Polyak’s heavy-ball method is renowned for accelerating convergence in convex settings, its convergence critically depends on regularity assumptions and/or fixed parameters $(\alpha, \beta)$ that often fail in non-convex scenarios (see, e.g., \cite[Section 4.6]{lessard2016analysis}). In the absence of convexity, constant stepsizes cannot guarantee global convergence. Moreover, most existing works addressing this issue either introduce safeguards concerning the descent property of the heavy-ball direction \cite{fan2023msl} or develop methods that are closely related to conjugate gradient approaches (see, e.g., \cite{lapucci2024globallyconvergentgradientmethod,Lee17,Liu2024,Powell1977}). To address these limitations, the work in \cite{donnini2025efficientglobalizationheavyballtype} proposes a novel advanced globalization strategy based on a curve search approach \cite{ben1990curved,botsaris1978differential,goldfarb1980curvilinear,gould2000exploiting,shi2005new,xu2016global}. By performing backtracking along suitable search curves to satisfy a (possibly non-monotone) Armijo-type condition, this framework ensures global convergence for iterative methods in non-convex settings. A key feature of this technique, often lost in other globalization strategies, is the ability to recover the behavior of the ``pure'' heavy-ball method in strongly convex cases.

Here, we propose an extension of the curve search method introduced in \cite{donnini2025efficientglobalizationheavyballtype} 
to the constrained context of problem \eqref{eq::con-min}. By leveraging curvilinear searches, we can employ 
Polyak's heavy-ball direction in settings that go beyond the unconstrained case, 
reverting to a feasible descent direction via an Armijo-type curve search whenever the update is not acceptable 
(e.g., due to infeasibility or failure to satisfy an Armijo-type condition). 

There are several motivations behind this study: (i) momentum methods for constrained problems remain relatively underexplored, despite their potential benefits; (ii) a natural extension of Polyak's heavy-ball scheme takes the form
\begin{equation*}
    x_{k+1} = \Pi_{\Omega}[x_k - \alpha \nabla f(x_k) + \beta ( x_k - x_{k-1})],
\end{equation*}
which has been studied in \cite{Tao22} in the context of non-smooth convex constrained problems; in this work, however, we focus on a different setting, where the smooth objective function may be nonlinear and, most importantly, nonconvex; (iii) our goal is to propose a momentum-based methodology that does not require significant additional computational effort; indeed, a possible alternative would be to perform a second projection to project the momentum direction $x_k - x_{k-1}$ onto the feasible set; however, avoiding such additional projections is particularly important in scenarios where computing the projection may incur a non-negligible cost.

The novel curve search algorithm is analyzed from a theoretical point of view, proving both its well-definedness and global convergence to stationary points for problem \eqref{eq::con-min}. We emphasize that, although these results are analogous to those obtained in the unconstrained case, they are far from trivial to establish, as we had to account for the effect of the momentum direction in a constrained setting, where it could potentially lead to infeasible iterates. 

The remainder of the paper is organized as follows. In Section \ref{sec::preliminaries}, we review the preliminary notions relevant to this work, including a description of the curve search method introduced in \cite{donnini2025efficientglobalizationheavyballtype} for the unconstrained case. In Section \ref{sec::proposals}, we present our extension of the curve search approach to the constrained setting, including the general scheme of the proposed method and its theoretical analysis. In Section \ref{sec::extensions}, we discuss several mechanisms that can be incorporated into the algorithm, such as a non-monotone line search, a spectral gradient choice of the directions defining the curve, and an adaptive version of the momentum direction suitable for constrained problems. In Section \ref{sec::experiments}, we report numerical experiments on a benchmark of problems with different types of feasible sets, highlighting the consistency and effectiveness of our approach by comparison with the state-of-the-art \texttt{SPG} algorithm. Finally, Section \ref{sec::conclusions} provides some concluding remarks.

\section{Preliminaries}
\label{sec::preliminaries}

In this section, we recall the definitions of the main concepts that will be used throughout the paper. We begin by defining the feasible directions at a point $\bar{x} \in \Omega$.

\begin{definition}
    \label{def::feasible_direction}
    Let $\bar{x} \in \Omega$. A vector $d \in \mathbb{R}^n$ is called a \textit{feasible direction} for $\Omega$ at $\bar{x}$ if there exists $\bar{\alpha} > 0$ such that
    \begin{equation*}
        \bar{x} + \alpha d \in \Omega, \quad \forall \alpha \in [0, \bar{\alpha}].
    \end{equation*}
\end{definition}

The notion of feasible directions allows us to define stationary points for constrained optimization problems.

\begin{definition}
    \label{def::stat}
    A point $\bar{x} \in \Omega$ is a \emph{stationary point} for problem \eqref{eq::con-min} if $$\nabla f(\bar{x})^\top(x-\bar{x}) \ge 0, \quad \forall x \in \Omega.$$
\end{definition}

We call \emph{feasible direction method} any algorithm that starts from a feasible point $x_0 \in \Omega$ and generates a sequence of feasible points $\{x_k\}$ according to 
\begin{equation*}
    x_{k+1}=x_k+\alpha_kd_k,    
\end{equation*}
where, if $x_k$ is not stationary, $d_k$ is a feasible \emph{descent} direction at $x_k$, i.e., $\nabla f(x_k)^\top d_k<0$, and the stepsize $\alpha_k>0$ is such that the feasibility and Armijo-type sufficient decrease conditions are satisfied in the new iterate: 
\begin{equation}
    \label{eq::armijo}
    \quad\quad
    \begin{aligned}
        \alpha_k & = \max_{h \in \mathbb{N}}\left\{\alpha_0\delta^h
        \middle\vert
        \begin{aligned}
            & x_k+\delta^h\alpha_0d_k\in\Omega\ \land\\
            &f(x_k+\delta^h\alpha_0d_k) \le f(x_k) + \sigma\delta^h\alpha_0\nabla f(x_k)^\top d_k
        \end{aligned}
        \right\}
    \end{aligned}
\end{equation}
with $\alpha_0 > 0, \delta \in (0, 1), \sigma \in (0, 1)$. The method stops whenever the current iterate $x_k$ is stationary.

In particular, we are interested in feasible direction methods where the search direction is not only a feasible descent direction, but it is also gradient-related, according to the following definition.

\begin{definition}[{\cite[Section 2.2.1]{bertsekas1999nonlinear}}]
\label{def:gr}
Let $\{x_k\}$ and $\{d_k\}$ be two sequences generated by a feasible direction method where $x_k \in \Omega$ and $d_k \in \mathbb{R}^n$ is a feasible descent direction at $x_k$, i.e., $\nabla f(x_k)^\top d_k < 0$. We say that the direction sequence $\{d_k\}$ is \textit{gradient-related} to $\{x_k\}$ if, for any subsequence $\{x_k\}_{k \in K}$ that converges to
a nonstationary point, the corresponding subsequence $\{d_k\}_{k \in K}$ is bounded and satisfies
$$\limsup_{k\to\infty, k \in K}\nabla f(x_k)^\top d_k < 0.$$
\end{definition}

We should note that many standard approaches for smooth constrained optimization problems actually rely on directions of this type:
\begin{itemize}
    \item projected gradient method \cite[Section 2.3]{bertsekas1999nonlinear}:
    \begin{equation*}
        d_k = \Pi_{\Omega}[x_k - \eta\nabla f(x_k)] - x_k,\text{ with }\eta > 0;
    \end{equation*}
    \item Frank-Wolfe method \cite[Section 2.2.2]{bertsekas1999nonlinear}:
    \begin{equation*}
        d_k = z_k - x_k\text{, with }z_k \in \argmin_{z \in \Omega}\nabla f(x_k)^\top(z-x_k);
    \end{equation*}
    note that, an additional assumption on the compactness of $\Omega$ is needed to guarantee the existence of a solution for the auxiliary problem;
    \item non-monotone spectral projected gradient method \cite{birgin2000nonmonotone,birgin01}, denoted as \texttt{SPG}:
    \begin{equation}
        \label{eq::d_SPG}
        d_k=\Pi_\Omega[x_k-\eta_k\nabla f(x_k)]-x_k,
    \end{equation}
    where the safeguarded ``inverse Rayleigh quotient'' $\eta_k$ is defined as $\eta_0 > 0$ and, for $k>0$,
    \begin{equation*}
    \eta_k=\min\left\{\eta_{\max},\max\left\{\eta_{\min},
    \frac{r_{k-1}^\top r_{k-1}}{r_{k-1}^\top y_{k-1}}\right\}\right\},
    \end{equation*}
    with $\eta_{\max} > \eta_{\min} > 0$, $r_{k-1}=x_k-x_{k-1}$, and $y_{k-1}=\nabla f(x_k)-\nabla f(x_{k-1})$. Note that \texttt{SPG} employs a non-monotone line search based on a safeguarded quadratic interpolation; we refer the reader to the cited references for further details. By incorporating these modifications into the classical projected gradient method, \texttt{SPG} is widely regarded as a state-of-the-art approach for smooth, convexly constrained optimization problems in which the Euclidean projection is computationally tractable.
\end{itemize}
We can easily verify that these directions are feasible directions for $\Omega$ at $x_k \in \Omega$ according to Definition~\ref{def::feasible_direction} with $\bar{\alpha} = 1$. Thus, note that, if $\alpha_0 = 1$ in the Armijo-type line search \eqref{eq::armijo}, the feasibility condition is trivially satisfied.

\subsection{The Curve Search Method for Unconstrained Optimization}

We now recall the concepts needed to define curve search methods. In particular, we consider the curve search method introduced for the unconstrained case in \cite{donnini2025efficientglobalizationheavyballtype}, focusing specifically on quadratic curves.

A curve search method iteratively defines the new iterate according to 
\begin{equation}
    \label{eq::curve_search}
    x_{k+1} = \gamma_k(t_k),
\end{equation}
where $\gamma_k: [0, 1] \to \mathbb{R}^n$ is a differentiable curve and $t_k>0$ is a parameter that determines how far to move along the curve. Note that, for a line-search type algorithm, $\gamma_k$ would define a straight line.
Representatives for this type of methods can be found in the literature (see, e.g., \cite{ben1990curved, xu2016global}).
In the methodology proposed in \cite{donnini2025efficientglobalizationheavyballtype}, search curves are parametrized as follows:
\begin{equation*}
    \gamma_k(t) = \gamma(t; x_k, d_k, \xi_k),
\end{equation*}
where $x_k \in \mathbb{R}^n$ is the current solution, $d_k \in \mathbb{R}^n$ is a search direction suitably related to the gradient at $x_k$, and $\xi_k \in \mathbb{R}^p$ represents an additional parameter, which, for instance, may denote a second search direction. Similarly to \cite{donnini2025efficientglobalizationheavyballtype}, the search curves must satisfy the following assumption.

\begin{assumption}[{\cite[Assumption 1]{donnini2025efficientglobalizationheavyballtype}}]\label{ass::defcurve}
    These conditions hold for the family of search curves $\gamma(\cdot;x,d,\xi)$: 
    \begin{enumerate}
        \item[(a)] $\gamma(t;x,d,\xi)$ is continuous w.r.t.\ $t,x,d$ and $\xi$;
        \item[(b)] $\gamma(0;x,d,\xi)=x$, i.e., the starting point of the curve is $x$;
        \item[(c)] the velocity $$\gamma'(t;x,d,\xi) = \left(\frac{\partial \gamma^1(t;x,d,\xi)}{\partial t},\ldots,\frac{\partial \gamma^n(t;x,d,\xi)}{\partial t}\right)^\top$$ exists and is continuous for all $t,x,d$ and $\xi$;
        \item[(d)] $\gamma'(0;x,d,\xi)=d$, i.e., the initial velocity of the curve is given by direction $d$.
    \end{enumerate}
\end{assumption}

The definition of descent directions can thus be generalized to descent curves.
\begin{definition}[{\cite[Definition~1]{donnini2025efficientglobalizationheavyballtype}}]
    \label{def::descent_curves}
    Let $x_k \in \mathbb{R}^n$ and $\gamma_k$ be a differentiable curve satisfying Assumption \ref{ass::defcurve}.
    We say that $\gamma_k$ is a \textit{descent curve} for $f$ at $x_k$ if $\gamma_k'(0)$ is a descent direction for $f$ at $x_k$, that is, if $\nabla f(x_k)^\top \gamma_k'(0) < 0$.
\end{definition}

Given a descent curve $\gamma_k$, we rely on the Armijo-type curve search proposed in \cite{xu2016global,donnini2025efficientglobalizationheavyballtype} to choose the stepsize $t_k$ in \eqref{eq::curve_search}:
\begin{equation}
    \label{eq::armijo-curve}
    t_k = \max_{h \in \mathbb{N}}\{\delta^h \mid f(\gamma_k(\delta^h)) \le f(x_k) + \sigma \delta^h \nabla f(x_k)^\top d_k\}.
\end{equation}
Under suitable assumptions, the curve search method is proven to converge to stationary points.

Although the method proposed in \cite{donnini2025efficientglobalizationheavyballtype} leaves some freedom in the definition of the curves, the main focus is on a specific class of curves that satisfy Assumption \ref{ass::defcurve}, namely polynomial curves of degree two:
\begin{equation}
    \label{eq::quadratic_gamma}
    \gamma_k(t)=\gamma(t;x_k,d_k,s_k)= x_k + t d_k + t^2 (s_k - d_k),
\end{equation}
where the primary direction $d_k$ is gradient-related and $\xi_k$ is replaced by a direction $s_k \in \mathbb{R}^n$, which may not be a descent direction and along which standard line-search techniques may be insufficient to guarantee convergence (e.g., Newton-type, heavy-ball or other momentum directions).
These curves satisfy Assumption \ref{ass::defcurve}. In particular, $\gamma(t;x_k,d_k,s_k)$ is continuous with respect to $t$, $x_k$, $d_k$ and $s_k$; $\gamma(0;x_k,d_k,s_k)=\gamma_k(0)=x_k$; the velocity $\gamma'(t;x_k,d_k,s_k)$ exists and is continuous with respect to $t$, $x_k$, $d_k$ and $s_k$; finally, the initial velocity is given by $\gamma'(0;x_k,d_k,s_k)=\gamma_k'(0)=d_k$.
Particular interest lies in choosing the heavy-ball direction as the second direction, i.e.,
\begin{equation*}
    s_k = -\alpha \nabla f(x_k) + \beta (x_k - x_{k-1}),
\end{equation*}
which makes the approach proposed in \cite{donnini2025efficientglobalizationheavyballtype} an effective and efficient globalization strategy for heavy-ball-type methods, capable of recovering the behavior of ``pure'' heavy-ball steps in the strongly convex setting.

These quadratic curves \eqref{eq::quadratic_gamma} can be rewritten in Bernstein basis to obtain Bézier curves of degree two \cite{farin2000essentials} corresponding to the points $(P_0,P_1,P_2)$:
\begin{equation}
    \label{eq::quadratic_bezier}
    \gamma_k(t) = (1-t)^2 P_0 + 2t(1-t) P_1 + t^2 P_2,
\end{equation}
with $P_0 = x_k$, $P_1 = x_k + \frac{1}{2} d_k$ and $P_2 = x_k + s_k$. We also refer the reader to Figure~1 in \cite{donnini2025efficientglobalizationheavyballtype} for a graphical illustration of this type of curve.
The Bézier formulation allows us to derive another crucial property of curves of the form \eqref{eq::quadratic_gamma}: for all $t \in [0,1]$, it holds that 
\begin{equation}
    \label{eq::gamma_CH}
    \gamma(t) \in \text{CH}(P_0,P_1,P_2),
\end{equation}
where $\text{CH}(P_0,P_1,P_2)$ denotes the convex hull of the points $(P_0,P_1,P_2)$. We conclude the section by introducing two additional properties, whose proofs are provided in Appendix~\ref{app::proof}.

\begin{lemma}
    \label{lemma::CHgamma}
    Let $\gamma:[0,1]\to\mathbb{R}^n$ be the quadratic Bézier curve corresponding to points $(P_0,P_1,P_2)$ and $\hat{t}\in[0,1]$. Then, for all $t\in[0,\hat{t}]$, $\gamma(t)\in \text{CH}(P_0,P_1,\gamma(\hat{t}))$. 
\end{lemma}

\begin{lemma}
    \label{lemma::CHinfeas}
    Let $\gamma:[0,1]\to\mathbb{R}^n$ be the quadratic Bézier curve corresponding to points $(P_0,P_1,P_2)$ with $P_0, P_1 \in \Omega$. If there exists $\hat{t}\in(0,1]$ such that $g_i(\gamma(\hat{t})) > 0$ for some $i \in [m]$, then $g_i(P_2) > 0$.
\end{lemma}

\section{A Curvilinear Descent Method for Constrained Optimization}
\label{sec::proposals}

In this section, we present an extension of the curve search method introduced in \cite{donnini2025efficientglobalizationheavyballtype} for smooth convexly constrained optimization.

We start by extending the notion of feasible directions to define feasible curves. Note that, similarly, the concept of feasible direction method can be easily extended to define feasible curve methods. Moreover, we adapt Definition~\ref{def:gr} to introduce gradient-related curves in the constrained setting.

\begin{definition}
\label{def::fc}
Let $x_k \in \Omega$ and let $\gamma_k$ be a differentiable curve satisfying Assumption~\ref{ass::defcurve}. We say that $\gamma_k$ is a \textit{feasible curve} for $\Omega$ at $x_k$ if there exists $\bar{t}\in(0,1]$ such that
\begin{equation*}
\gamma_k(t) \in \Omega, \quad \forall t \in [0, \bar{t}].
\end{equation*}
\end{definition}

\begin{definition}
\label{def::gradient-related}
Let $\{x_k\}$ and $\{\gamma_k\}$ be two sequences generated by a \textit{feasible curve method},
where $x_k \in \Omega$ and $\gamma_k$ is a feasible \textit{descent} curve at $x_k$, i.e., $\nabla f(x_k)^\top \gamma_k'(0) < 0$. We say that the curve sequence $\{\gamma_k\}$ is \textit{gradient-related} to $\{x_k\}$ if, for any subsequence $\{x_k\}_{k \in K}$ that converges to a nonstationary point, the corresponding subsequence $\{\gamma_k'(0)\}_{k \in K}$ is bounded and satisfies
\begin{equation*}
\limsup_{k\to\infty, k \in K} \nabla f(x_k)^\top \gamma_k'(0) < 0.
\end{equation*}
\end{definition}

Similarly to \cite{donnini2025efficientglobalizationheavyballtype}, we consider curves $\gamma_k(t)=\gamma(t;x_k,d_k,s_k)$ of the form \eqref{eq::quadratic_gamma}, where $d_k$ is a feasible gradient-related direction and $s_k$ is the Polyak's heavy-ball direction adapted to the constrained setting, i.e.,
\begin{equation}
\label{eq::polyak_constrained}
s_k = \alpha d_k + \beta (x_k - x_{k-1}).
\end{equation}
It is straightforward to verify that the curves $\gamma_k$ are gradient-related in the sense of Definition~\ref{def::gradient-related}.

In the following, we assume that the direction $d_k$ satisfies the following assumption.

\begin{assumption}
\label{ass::dk_feas}
The direction $d_k$ is a feasible direction according to Definition~\ref{def::feasible_direction} with $\bar{\alpha} = 1$.
\end{assumption}

Note that this assumption is reasonable in the constrained setting, as many standard directions satisfy it, as already mentioned in Section~\ref{sec::preliminaries}. This also implies that we consider quadratic curves with control points $P_0 = x_k \in \Omega$ and $P_1 = x_k + \frac{1}{2} d_k \in \Omega$.

In Figure~\ref{fig::example-cs}, we show a graphical example of a feasible curve of the form \eqref{eq::quadratic_gamma} satisfying Assumption~\ref{ass::dk_feas}. Note that we do not require $P_2 \in \Omega$ for the curve to be feasible. If $P_2 \in \Omega$, then by equation \eqref{eq::gamma_CH} and convexity of $\Omega$ every point of the curve is feasible, i.e., the curve is feasible according to Definition~\ref{def::fc} with $\bar{t} = 1$.

\begin{figure}[ht]
\centering
\includegraphics[width=0.6\textwidth]{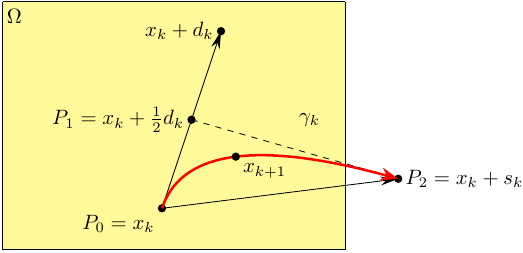}
\caption{Feasible curve $\gamma_k$ of the form \eqref{eq::quadratic_gamma} satisfying Assumption~\ref{ass::dk_feas}. The convex feasible set $\Omega$ is highlighted in yellow.}
\label{fig::example-cs}
\end{figure}

In general, a quadratic curve of the form \eqref{eq::quadratic_gamma} must satisfy additional reasonable conditions in order to be feasible. These conditions are formalized in the following proposition.

\begin{proposition}
\label{prop::feasibility}
Let $x_k \in \Omega$ and $\gamma_k$ be a gradient-related curve of the form \eqref{eq::quadratic_gamma} satisfying Assumptions~\ref{ass::defcurve}-\ref{ass::dk_feas}. Let $\tilde{t} \in (0,1)$, $\tilde{x}_k = x_k + \tilde{t} d_k$ and $\tilde{I}_k = \{i \in [m] \mid g_i(\tilde{x}_k) = 0\}$ denote the set of active constraints at $\tilde{x}_k$. If $g_i(x_k + s_k) \le 0$ for all $i \in \tilde{I}_k$, then $\gamma_k$ is a feasible curve for $\Omega$ at $x_k$.
\end{proposition}

\begin{proof}
    First, we consider the case $P_2 = x_k + s_k \in \Omega$. Given that $x_k \in \Omega$, $\Omega$ is convex, Assumption \ref{ass::dk_feas} and equations \eqref{eq::quadratic_bezier}-\eqref{eq::gamma_CH} hold, it follows that $\gamma_k(t) \in \text{CH}(P_0, P_1, P_2) \subseteq \Omega$ for all $t \in [0, 1]$. Consequently, we have that the curve is feasible according to Definition \ref{def::fc} with $\bar{t} = 1$.

    We now consider the case $P_2 = x_k+s_k\not\in\Omega$. Let us define the composite function $h^k: \mathbb{R}\to\mathbb{R}^m$ as $h^k(t)=g(\gamma_k(t))$; since $g$ and $\gamma_k$ (see Assumption \ref{ass::defcurve}) are continuously differentiable,  $h^k$ is also continuously differentiable.
    Furthermore, let us define the set 
    \begin{equation*}
        I_k = \{i \in [m] \mid g_i(x_k) = 0\}    
    \end{equation*}
    as the set of active constraints in $x_k$. In order to prove that there exists $\bar{t}\in(0,1)$ such that Definition~\ref{def::fc} is satisfied, we consider separate subcases, partitioning the set of constraints $[m]$ into the sets $[m] \setminus I_k$, $\tilde{I}_k$ and  $I_k \setminus \tilde{I}_k$. Note that $\Omega = \Omega_{[m] \setminus I_k} \cap \Omega_{\tilde{I}_k} \cap \Omega_{I_k \setminus \tilde{I}_k}$. In the following, we will show that for each constraints set $I\in\{[m] \setminus I_k,\tilde{I}_k,I_k \setminus \tilde{I}_k\}$ there exists a value $\bar t_I\in(0,1)$ such that $\gamma_k(t)\in \Omega_I,\ \forall t\in[0,\bar t_I]$. Finally, by taking $\bar t=\min_{I\in \{[m] \setminus I_k,\tilde{I}_k,I_k \setminus \tilde{I}_k\}}\bar t_I$, it will hold that 
    \begin{equation*}
        \gamma_k(t)\in\Omega_{[m] \setminus I_k} \cap \Omega_{\tilde{I}_k} \cap \Omega_{I_k \setminus \tilde{I}_k}=\Omega, \quad \forall t\in[0,\bar t].
    \end{equation*}
    
    We now define the values of $\bar t_{[m]\setminus I_k}, \bar t_{\tilde I_k}$ and $\bar t_{I_k\setminus\tilde I_k}$, one at a time.
    \begin{enumerate}
        \item We further divide the set $[m] \setminus I_k$ into 
        \begin{equation*}
            \bar{I}_k^1 = \{i \in [m] \setminus I_k \mid g_i(x_k+s_k) \le 0\} \text{ and } \bar{I}_k^2 = \{i \in [m] \setminus I_k \mid g_i(x_k+s_k) > 0\}.
        \end{equation*}
        \begin{enumerate}
            \item By $x_k \in \Omega$, Assumption \ref{ass::dk_feas}, definition of $\bar{I}_k^1$ and $\Omega \subseteq \Omega_{\bar{I}_k^1}$, we have that $P_0 = x_k \in \Omega_{\bar{I}_k^1}$, $P_1 = x_k + \frac{1}{2}d_k \in \Omega_{\bar{I}_k^1}$ and $P_2 = x_k+s_k \in \Omega_{\bar{I}_k^1}$. Thus, by equations \eqref{eq::quadratic_bezier}-\eqref{eq::gamma_CH} and the convexity of $\Omega_{\bar{I}_k^1}$, we get that $\gamma_k(t) \in \text{CH}(P_0,P_1,P_2) \subseteq\Omega_{\bar{I}_k^1}$ for all $t \in [0, 1]$. We thus set $\bar{t}_{\bar{I}_k^1} = 1$. 
            
            \item Since Assumption \ref{ass::defcurve} holds and $x_k \in \Omega$ it follows that $h_i^k(0)=g_i(\gamma_k(0)) = g_i(x_k) < 0$ for all $i \in \bar{I}_k^2$, while by definition of $\bar{I}_k^2$ it follows that $h_i^k(1)=g_i(\gamma_k(1))=g_i(x_k+s_k)>0$  for all $i \in \bar{I}_k^2$. By the sign-preservation property of continuous functions there thus exists a value $\hat{t}_i \in (0,1)$ for all $i \in \bar{I}_k^2$ such that $h_i^k(t)<0$ for all $t\in[0,\hat{t}_i]$. Accordingly, we define $\bar{t}_{\bar{I}_k^2}=\min_{i \in \bar{I}_k^2}\hat{t}_i.$
        \end{enumerate}

        We conclude the case setting 
        \begin{equation*}
            \bar{t}_{[m] \setminus I_k} = \min\left\{\bar{t}_{\bar{I}_k^1}, \bar{t}_{\bar{I}_k^2}\right\} > 0.
        \end{equation*}
        
        \item By $x_k \in \Omega$, Assumption \ref{ass::dk_feas}, definition of $\tilde{I}_k$ and $\Omega \subseteq \Omega_{\tilde{I}_k}$, we have that $P_0 = x_k \in \Omega_{\tilde{I}_k}$, $P_1 = x_k + \frac{1}{2}d_k \in \Omega_{\tilde{I}_k}$. Moreover, we know by hypothesis that $g_i(x_k + s_k) \le 0$ for all $i \in \tilde{I}_k$, i.e., $x_k + s_k \in \Omega_{\tilde{I}_k}$. Then, by equations \eqref{eq::quadratic_bezier}-\eqref{eq::gamma_CH} and the convexity of $\Omega_{\tilde I_k}$, it follows that $\gamma_k(t) \in \text{CH}(P_0, P_1, P_2) \subseteq \Omega_{\tilde{I}_k}$ for all $t \in [0,1]$. We thus define 
        \begin{equation*}
            \bar{t}_{\tilde{I}_k} = 1.
        \end{equation*}
        
        \item By Assumption \ref{ass::defcurve}, $x_k \in \Omega$, definition of $I_k \setminus \tilde{I}_k$ and $\Omega \subseteq \Omega_{I_k \setminus \tilde{I}_k}$, we have that $P_0 = x_k \in \Omega_{I_k \setminus \tilde{I}_k}$ and, in particular,
        \begin{equation}
            \label{eq::propertyxk}
            h_i^k(0) = g_i(\gamma_k(0)) = g_i(x_k) = 0, \quad \forall i \in I_k \setminus \tilde{I}_k.
        \end{equation}
        Let us assume by contradiction that $\gamma_k(t) \not \in \Omega_{I_k \setminus \tilde{I}_k}$ for all $t \in (0, 1]$, that is, there exists an index $i_t \in I_k \setminus \tilde{I}_k$ for all $t \in (0, 1]$ such that $h_{i_t}^k(t) > 0$. Combining this assumption with \eqref{eq::propertyxk}, we can write
        \begin{equation*}
            \frac{h_{i_t}^k(t)-h_{i_t}^k(0)}{t}>0, \quad \forall t \in (0, 1],
        \end{equation*}
        which in turn implies
        \begin{equation}
            \label{eq::before_limit}
            \max_{i\in I_k \setminus \tilde{I}_k}\frac{h_{i}^k(t)-h_{i}^k(0)}{t}>0, \quad \forall t \in (0, 1].
        \end{equation}
        Since $h$ is continuously differentiable and the set $I_k \setminus \tilde{I}_k$ is finite, we can thus take the limit in \eqref{eq::before_limit} for $t \to 0^+$ to obtain 
        \begin{equation}
            \label{eq::first_contradiction}
            \begin{aligned}
                0& \leq\lim_{t\to 0^+}\left(\max_{i\in I_k \setminus \tilde{I}_k}\frac{h_{i}^k(t)-h_{i}^k(0)}{t}\right) = \max_{i\in I_k \setminus \tilde{I}_k}\left(\lim_{t\to0^+}\frac{h_i^k(t)-h_i^k(0)}{t}\right) \\ & = \max_{i\in I_k \setminus \tilde{I}_k} (h_i^k)'(0)=(h_{i^\star}^k)'(0)=\nabla g_{i^\star}(\gamma_k(0))^\top\gamma_k'(0) = \nabla g_{i^\star}(x_k)^\top d_k,
            \end{aligned}
        \end{equation}
        where $i^*\in I_k\setminus\tilde I_{k}$ is such that $i^*\in\arg\max_{i\in I_k\setminus\tilde I_k}(h_i^k)'(0)$ and
        the last equality holds by Assumption \ref{ass::defcurve}.

        By Assumption \ref{ass::dk_feas}, definition of $I_k \setminus \tilde{I}_k$ and $\Omega \subseteq \Omega_{I_k \setminus \tilde{I}_k}$, we have that $ \tilde{x}_k \in \Omega_{I_k \setminus \tilde{I}_k}$ and, in particular,
        \begin{equation}
            \label{eq::propertytildexk}
            g_{i^\star}(\tilde{x}_k) < 0.
        \end{equation}
        The continuous differentiability and convexity of $g_{i^\star}$ also implies that 
        \begin{equation}
            \label{eq::cdcg}
            g_{i^\star}(\tilde{x}_k) \ge g_{i^\star}(x_k) + \nabla g_{i^\star}(x_k)^\top(\tilde{x}_k - x_k).
        \end{equation}
        Thus, we can write 
        \begin{equation*}
            0 > g_{i^\star}(\tilde{x}_k)\geq \nabla g_{i^\star}(x_k)^\top(\tilde{x}_k - x_k) = \tilde{t}\nabla g_{i^\star}(x_k)^\top d_k,
        \end{equation*}
        where the first inequality follows from \eqref{eq::propertytildexk}, the second one is obtained combining equations \eqref{eq::propertyxk} and \eqref{eq::cdcg}, and the last equality is derived by definition of $\tilde{x}_k$.
        
        Since $\tilde{t} > 0$, the last result is in contradiction with equation \eqref{eq::first_contradiction}. Then, we conclude that 
        \begin{equation*}
            \exists \bar{t}_{I_k \setminus \tilde{I}_k} \in (0, 1] \text{ s.t.\ } \gamma_k\left(\bar{t}_{I_k \setminus \tilde{I}_k}\right) \in \Omega_{I_k \setminus \tilde{I}_k}.
        \end{equation*}
        Moreover, by equation \eqref{eq::propertyxk}, Assumption \ref{ass::dk_feas}, definition of $I_k \setminus \tilde{I}_k$ and $\Omega \subseteq \Omega_{I_k \setminus \tilde{I}_k}$ we note that $P_0 = x_k \in \Omega_{I_k \setminus \tilde{I}_k}$ and $P_1 = x_k + \frac{1}{2}d_k \in \Omega_{I_k \setminus \tilde{I}_k}$ Thus, it follows by the convexity of $\Omega_{I_k \setminus \tilde{I}_k}$ that $CH\left(P_0,P_1,\gamma_k\left(\bar{t}_{I_k \setminus \tilde{I}_k}\right)\right)\subseteq\Omega_{I_k \setminus \tilde{I}_k}$. By Lemma \ref{lemma::CHgamma} we finally get that 
        \begin{equation*}
            \gamma_k(t) \in \text{CH}\left(P_0, P_1, \gamma_k\left(\bar{t}_{I_k \setminus \tilde{I}_k}\right)\right) \subseteq \Omega_{I_k \setminus \tilde{I}_k}, \quad \forall t \in \left[0, \bar{t}_{I_k \setminus \tilde{I}_k}\right].
        \end{equation*}
    \end{enumerate}
    Setting $\bar{t}=\min\left\{\bar{t}_{[m] \setminus I_k},\bar{t}_{\tilde{I}_k},\bar{t}_{I_k \setminus \tilde{I}_k}\right\} > 0$ concludes the proof.
\end{proof}

In the next example, we show that the final hypothesis of Proposition \ref{prop::feasibility}, that is, $g_i(x_k+s_k)\leq0$ for all $i\in\tilde I_k$, is indeed necessary to ensure the feasibility of the curve.

\begin{example}
\label{ex::example}

Figure~\ref{fig::example} illustrates two case studies of quadratic Bézier curves of the form \eqref{eq::quadratic_gamma} applied to a two-dimensional toy problem subject to a single linear constraint defining the feasible set 
\begin{equation*}
    \Omega = \{x=[x^1,x^2]^\top \in \mathbb{R}^2 \mid x^1 \le 0\},    
\end{equation*}
i.e., $g_1(x)=x^1$. Let $x_k = [0, 0]^\top$ and $d_k = [0, 4]^\top$. The first two control points of the curve are then $P_0 = x_k$ and $P_1 = x_k + \frac{1}{2}d_k = [0, 2]^\top$. According to Proposition \ref{prop::feasibility}, let us consider $\tilde t=\frac{1}{4}$ and define $\tilde{x}_k = x_k+\frac{1}{4}d_k$. The constraint $g_1(x)\le 0$ is active at $\tilde{x}_k$ and, thus, $\tilde{I}_k = \{1\}$. 

\begin{figure}[ht]
     \centering
     \subfloat[The third control point $P_2$ is feasible with respect to the constraint active at $\tilde{x}_k$. Thus, $\gamma_k$ is a feasible curve.]{\includegraphics[width=0.49\textwidth]{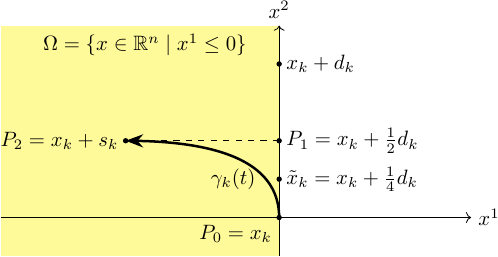}}
     \hfill
     \subfloat[The third control point $P_2$ is infeasible with respect to the constraint active at $\tilde{x}_k$. Here, $\gamma_k$ is not a feasible curve.]{\includegraphics[width=0.49\textwidth]{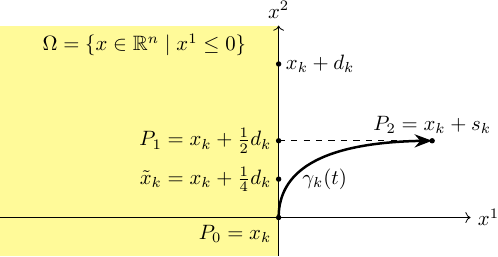}}
     \caption{Two case studies of quadratic curves of the form \eqref{eq::quadratic_gamma} on the two-dimensional problem with one constraint described in Example \ref{ex::example}.}
     \label{fig::example}
\end{figure}

We consider two distinct cases for the direction $s_k$. In the first scenario shown on the left in Figure~\ref{fig::example}, $s_k = [-4, 2]^\top$, which results in $P_2 = x_k+s_k=[-4,2]^\top$ and $g_1(x_k+s_k)\leq0$, i.e., the final hypothesis of Proposition \ref{prop::feasibility} is satisfied. It thus follows that the curve $\gamma_k$ with control points $P_0,P_1$ and $P_2$ is feasible. In the second scenario shown on the right in Figure~\ref{fig::example}, $s_k = [4, 2]^\top$. In this case, point $P_2$ violates the active constraint and thus fails to satisfy the final hypothesis of Proposition \ref{prop::feasibility}. The violation of the hypothesis implies that the curve is infeasible, i.e., there exists no $\bar{t} \in (0, 1]$ such that $\gamma_k(t) \in \Omega$ for all $t \in [0, \bar{t}]$. This result can be further formalized from a mathematical perspective (analogous reasoning can be applied to the first scenario). The equation of curve $\gamma_k$ can be written following equation \eqref{eq::quadratic_bezier} as follows:
\begin{equation*}
    \gamma_k(t) = (1-t^2)\begin{bmatrix}
        0\\0
    \end{bmatrix} + 2t(1-t) \begin{bmatrix}
        0\\2
    \end{bmatrix} + t^2 \begin{bmatrix}
        4\\2
    \end{bmatrix} = \begin{bmatrix}
        4t^2\\4t - 2t^2
    \end{bmatrix}.
\end{equation*}
It is thus easy to verify that $g_1(\gamma_k(t)) = 4t^2 > 0$, for all $t \in (0, 1]$, which implies that $\gamma_k$ is not a feasible curve for $\Omega$ at $x_k$. In such a scenario, we cannot rely on curve search methods and we must fall back on standard line searches; further details will be provided in Section~\ref{subsec::algorithmic_scheme}.
\end{example}

\begin{remark}
    Although the final hypothesis of Proposition~\ref{prop::feasibility} on the point $P_2 = x_k + s_k$ may appear strong or somewhat cryptic for the use of curve searches on constrained problems, it is in fact necessary to handle the specific situation in which the current iterate $x_k$ lies on the boundary of a linear constraint and the gradient-related direction $d_k$ suggests moving along that constraint boundary (as illustrated in Example \ref{ex::example}). However, outside of this specific scenario the assumption is not particularly restrictive and is indeed automatically satisfied by a wide class of convex feasible regions characterized by curved boundaries. In particular, any convex feasible set whose boundary does not contain a segment (e.g., spherical or ellipsoidal) satisfies said assumption, since the set $\tilde I_k$ is always empty as $\tilde x_k$ is a point that belongs to the interior of $\Omega$ by construction. 
\end{remark}

\subsection{Algorithmic Scheme}
\label{subsec::algorithmic_scheme}

In Algorithm \ref{alg::CSM_constrained}, we provide the algorithmic scheme of our curve search method for smooth convexly constrained optimization problems.

\begin{algorithm}[ht]\caption{Curve Search method for smooth convexly constrained problems}
	\label{alg::CSM_constrained}
	\begin{algorithmic}[1]
		\REQUIRE 
			$x_0\in\Omega$, $\alpha, \beta > 0$, $\tilde{t} \in (0, 1)$, $\delta \in (0, 1)$, $\sigma \in (0, 1)$, $\{\varepsilon_k\} \subseteq \mathbb{R}_+$ a decreasing sequence.
		\STATE Set $k=0$
        \STATE Set $x_{-1}=x_0$
        \WHILE{stopping criterion is not satisfied\label{line::while}}
        \STATE Compute feasible gradient-related direction $d_k$\label{line::d_k}
        \STATE Compute direction $s_k = \alpha d_k + \beta(x_k - x_{k-1})$\label{line::s_k}
        \STATE Compute set $\tilde{I}_k=\{i\in[m]|g_i(x_k + \tilde{t}d_k) \ge -\varepsilon_k\}$\label{line::Ik}
        \IF{$\exists i\in \tilde{I}_k: g_i(x_k+s_k) > 0$ \textbf{or} $k = 0$ \label{line::if_restore}}
        \STATE Set $s_k = d_k$ \label{line::restore}
        \ENDIF
        \STATE Let $\gamma_k(t)=x_k + td_k + t^2(s_k - d_k)$\label{line::gamma}
        \STATE Compute $t_k = \max_{h \in \mathbb{N}}\{\delta^h\mid \gamma_k(\delta^h) \in \Omega \land f(\gamma_k(\delta^h))\le f(x_k)+\sigma \delta^h \nabla f(x_k)^\top d_k\}$\label{line::armijo}
        \STATE Set $x_{k+1}=\gamma_k(t_k)$\label{line::next_iterate}
        \STATE Set $k = k + 1$
        \ENDWHILE
		\RETURN $x_k$\label{line::return}
	\end{algorithmic}
\end{algorithm}

Starting from a feasible point $x_0 \in \Omega$, at each iteration we compute in line \ref{line::d_k} a feasible gradient-related direction $d_k$ and in line \ref{line::s_k} the Polyak's heavy-ball direction $s_k$, adapted to the constrained case (see also equation \eqref{eq::polyak_constrained}).

In order to define the curve in line \ref{line::gamma}, we first check if a feasible curve can be constructed. According to Proposition \ref{prop::feasibility}, we verify whether the point $x_k + s_k$ is feasible with respect to the active constraints at the point $\tilde{x}_k = x_k + \tilde{t}d_k$ (lines \ref{line::Ik}--\ref{line::if_restore}). If this condition is not satisfied (or if we are at the first iteration of the algorithm), we set $s_k = d_k$, where $d_k$ is the gradient-related direction. Note that, with this choice, we obtain
\begin{equation*}
    \gamma_k(t) = x_k + td_k + t^2(d_k - d_k) = x_k + td_k,
\end{equation*}
that is, the curve reduces to a straight line. It is easy to verify that, in this case, the assumption of Proposition \ref{prop::feasibility} is automatically satisfied, since $g_i(x_k + d_k) \leq 0$ for each $i \in [m]$, as we assume that the gradient-related direction $d_k$ satisfies Assumption \ref{ass::dk_feas}. Otherwise, the curve is defined as in equation~\eqref{eq::quadratic_gamma}, which corresponds to the standard behavior of the algorithm.

Figure~\ref{fig::example_curves} illustrates the behavior of the algorithm in the two scenarios described in Example~\ref{ex::example}.

\begin{figure}[ht]
    \subfloat[Since $P_2$ is feasible with respect to the active constraints at $\tilde{x}_k$, the quadratic curve $\gamma_k$ of the form \eqref{eq::quadratic_gamma} is employed for the Armijo-type curve search.]{\includegraphics[width=0.49\textwidth]{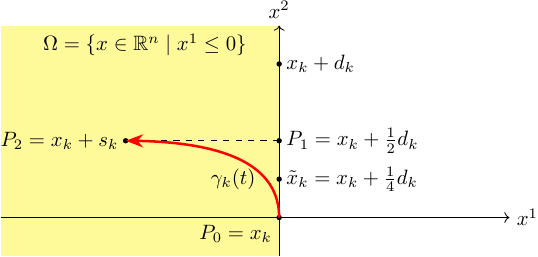}}
    \hfill
    \subfloat[When $P_2$ is infeasible with respect to the active constraints at $\tilde{x}_k$, the algorithm reverts to the direction $d_k$, thus employing a classical Armijo-type line search.]{\includegraphics[width=0.49\textwidth]{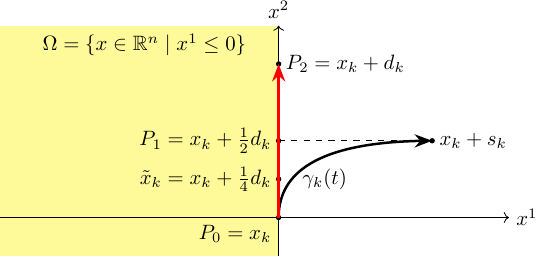}}
     \caption{Graphical representation of the curves considered by Algorithm \ref{alg::CSM_constrained} in the two scenarios described in Example \ref{ex::example}.}
     \label{fig::example_curves}
\end{figure}

Note that, for convergence reasons, at each iteration $k$ of the algorithm the set $\tilde{I}_k$ is computed in line \ref{line::Ik} so as to contain the indices of the constraints that are active at $\tilde{x}_k$ and those that are nearly active within a tolerance $\varepsilon_k$. This tolerance is decreased at each iteration of the algorithm in order to progressively focus only on the constraints that are active at $\tilde{x}_k$.

In line \ref{line::armijo} we perform the Armijo-type curve search described in equation \eqref{eq::armijo-curve} with the additional requirement that the next iterate (line \ref{line::next_iterate}) must be feasible. Whenever line \ref{line::restore} is executed, that is, $s_k = d_k$ and $\gamma_k$ reduces to a straight line, the Armijo-type curve search reduces to a standard line search, with the feasibility condition already satisfied since $d_k$ satisfies Assumption \ref{ass::dk_feas}.

The algorithm proceeds until a suitable stopping criterion (e.g., stationarity) is met (line \ref{line::while}), returning the current iterate $x_k$ (line \ref{line::return}).

\subsection{Convergence Analysis}

In this section, we present a theoretical analysis of Algorithm \ref{alg::CSM_constrained}.

We start by proving the finite termination of the Armijo-type curve search in line \ref{line::armijo}, which in turn ensures that Algorithm \ref{alg::CSM_constrained} is well-defined.

\begin{proposition}
    \label{prop::ACSprop}
    Let $x_k \in \Omega$ and $\gamma_k$ be a gradient-related curve of the form \eqref{eq::quadratic_gamma} satisfying Assumptions \ref{ass::defcurve}-\ref{ass::dk_feas}. Then, the Armijo-type curve search procedure in line \ref{line::armijo} of Algorithm \ref{alg::CSM_constrained} terminates in a finite number of iterations returning a valid stepsize $t_k > 0$. Moreover, one of the two following conditions holds:
    \begin{enumerate}
        \item[(i)] $t_k=1$;
        \item[(ii)] if $t_k<1$, then $\gamma_k(\frac{t_k}{\delta}) \not \in \Omega$ or $f(\gamma_k(\frac{t_k}{\delta}))>f(x_k)+\sigma \frac{t_k}{\delta}\nabla f(x_k)^\top d_k$.
    \end{enumerate}
\end{proposition}
\begin{proof}
    We distinguish between the case in which $s_k$ is defined in line \ref{line::s_k} and the case in which it is overwritten by setting $s_k = d_k$ in line \ref{line::restore}.
    
    If $s_k = d_k$, it follows from Assumption \ref{ass::dk_feas} that the feasibility condition in line \ref{line::armijo} is already satisfied. The Armijo-type curve search thus reduces to a standard line search, for which finite termination results are already known \cite[Section 20.2]{grippo2023introduction}.
    
    Otherwise, if $s_k$ is defined according to line \ref{line::s_k}, we know from \cite[Proposition 1]{donnini2025efficientglobalizationheavyballtype} that, in a finite number of iterations, we can find $\hat{t}_k > 0$ such that the Armijo sufficient decrease condition is satisfied for any $t \in (0, \hat{t}_k]$. Furthermore, since in this case $s_k$ is not overwritten in line \ref{line::restore}, the curve $\gamma_k$ satisfies the assumptions of Proposition \ref{prop::feasibility}, which implies that $\gamma_k$ is a feasible curve. Hence, there exists $\tilde{q} \in \mathbb{N}$ such that $\tilde{t}_k = \delta^{\tilde{q}} > 0$ and $\gamma_k(t) \in \Omega$ for all $t \in [0, \tilde{t}_k]$.
    We thus set $t_k = \min\{\hat{t}_k, \tilde{t}_k\} > 0$, so that both the Armijo sufficient decrease and the feasibility conditions are satisfied for all $t \in (0, t_k]$.
        
    Now, if no backtracking is needed, condition (i) straightforwardly holds. On the other hand, if at least one backtracking step is performed, then the stepsize $\frac{t_k}{\delta}$ is the last tested step not satisfying at least one of the Armijo-type conditions in line \ref{line::armijo} of Algorithm \ref{alg::CSM_constrained}; thus, condition (ii) holds. This concludes the proof.
\end{proof}

We continue by characterizing the points produced by our curve search method.

\begin{proposition}
    \label{prop::feasibility2}
    Let $\{(x_k,\gamma_k)\}$ be the sequence generated by Algorithm \ref{alg::CSM_constrained} where, for all $k$, $\gamma_k$ is a gradient-related curve satisfying Assumption \ref{ass::dk_feas}. Then, we have that $\{x_k\} \subseteq \Omega$.
\end{proposition}
\begin{proof}
    By the instructions of Algorithm~\ref{alg::CSM_constrained} and Assumption~\ref{ass::dk_feas}, at iteration $k = 0$ we have $x_0 \in \Omega$ and the curve $\gamma_0$ is a feasible descent curve. Therefore, by line~\ref{line::armijo} of Algorithm~\ref{alg::CSM_constrained} and Proposition~\ref{prop::ACSprop}, we obtain that $x_1 = \gamma_0(t_0) \in \Omega$. The thesis then follows by induction.
\end{proof}

Before stating the global convergence result of the proposed curve search method, we must introduce the following reasonable assumptions, similar to those used in \cite{donnini2025efficientglobalizationheavyballtype}.

\begin{assumption}
    \label{ass::L0}
    Let $x_0 \in \Omega$. The level set $\mathcal{L}_0 = \{x \in \Omega \mid f(x) \le f(x_0)\}$ is compact.
\end{assumption}

\begin{assumption}
    \label{ass::xi}
    Let $\{(x_k, \gamma_k)\}$ be the sequence generated by Algorithm \ref{alg::CSM_constrained}. The sequence $\{s_k\}$ is bounded, i.e., $\exists C > 0$ such that $\|s_k\| \le C$ for all $k$.
\end{assumption}

\begin{proposition}\label{prop:globalconvm}
    Let Assumption \ref{ass::L0} hold. Let $\{(x_k,\gamma_k)\}$ be the sequence generated by Algorithm \ref{alg::CSM_constrained} such that:
    \begin{enumerate}
        \item for all $k$, $\gamma_k$ is a gradient-related curve satisfying Assumptions \ref{ass::defcurve}-\ref{ass::dk_feas};
        \item Assumption \ref{ass::xi} holds for sequence $\{s_k\}$. 
    \end{enumerate}  
    Then, the sequence $\{x_k\}$ admits accumulation points and each accumulation point $\bar{x}$ is stationary for problem \eqref{eq::con-min}.
\end{proposition}
\begin{proof}
    By Proposition \ref{prop::feasibility}, we have that $\{x_k\} \subseteq \Omega$. Moreover, by hypothesis, the sequence $\{\gamma_k\}$ is gradient-related to $\{x_k\}$. Thus, we get by the Armijo-type sufficient decrease condition in line \ref{line::armijo} of Algorithm \ref{alg::CSM_constrained} that, for all $k$,
    \begin{equation*}
        f(x_{k+1})=  f(\gamma_k(t_k)) \le f(x_k) + \sigma t_k \nabla f(x_k)^\top d_k < f(x_k),
    \end{equation*}
    where the last inequality follows from Definition~\ref{def::gradient-related}. This last result implies that $\{f(x_k)\}$ is monotone decreasing. Hence, 
    \begin{equation}
        \label{eq::xkL0}
        \{x_k\} \subseteq \mathcal{L}_0 \subseteq \Omega.
    \end{equation}
    
    Since $\mathcal{L}_0$ is compact by Assumption \ref{ass::L0}, it follows that $\{x_k\}$ admits accumulation points, each one belonging to $\mathcal{L}_0$. Furthermore, the sequence $\{f(x_k)\}$ converges to some finite value $f^\star$ and 
    \begin{equation}
        \label{limfk}
        \lim_{k\to\infty} f(x_{k+1})-f(x_k)=0.
    \end{equation}
    Let $\bar{x}$ be one of the accumulation points, i.e., there exists a sub-sequence $K\subseteq\{0,1,\dots\}$ such that $\lim_{k\in K,k\to\infty} x_k=\bar{x}$. Let us assume by contradiction that $\bar{x}$ is not stationary.
  
    By the Armijo sufficient decrease condition in line \ref{line::armijo} of Algorithm \ref{alg::CSM_constrained}, we know that, for all $k$,
    \begin{equation*}
        f(x_{k+1})-f(x_k)\leq \sigma t_k\nabla f(x_k)^\top d_k.
    \end{equation*}
    Recalling equation \eqref{limfk}, we can take the limit in the previous inequality: 
    \begin{equation*}
        \lim_{k\in K,k\to\infty} \sigma t_k\nabla f(x_k)^\top d_k \ge 0.
    \end{equation*}
    Given $\sigma \in (0, 1)$, $t_k \in [0, 1]$ and, by definition of $d_k$, $\nabla f(x_k)^\top d_k \le 0$, the last result yields $\lim_{k\in K,k\to\infty} t_k\nabla f(x_k)^\top d_k = 0$. Since $\{\gamma_k\}$ is gradient-related to $\{x_k\}$ (Definition~\ref{def::gradient-related}) and the subsequence $\{x_k\}_{k \in K}$ converges to a non-stationary point, it then follows that 
    \begin{equation*}
        \lim_{k \in K, k \to \infty} t_k = 0.
    \end{equation*}
    
    The last limit implies that, for any given $q\in\mathbb{N}$, we have for $k \in K$ sufficiently large $t_k<\delta^q$. It thus follows by Proposition \ref{prop::ACSprop} that the stepsize $t = \delta^q$ does not satisfy at least one of the Armijo conditions in line \ref{line::armijo} of Algorithm \ref{alg::CSM_constrained}: 
    \begin{equation}
        \label{eq::armijo_unsatisfied}
        f\left(\gamma_k\left(\delta^q\right)\right) > f(x_k) + \sigma\delta^q\nabla f(x_k)^\top d_k \quad \text{ or } \quad \gamma_k(\delta^q) \not \in \Omega.
    \end{equation}
    We now consider three possible subsequences, one at a time.
    \begin{enumerate}
        \item[(i)] $\bar{K}\subseteq K$ where, for all $k \in \bar{K}$, $\gamma_k$ is a curve (line \ref{line::restore} of Algorithm \ref{alg::CSM_constrained} is thus not executed) and, for all $k \in \bar{K}$ sufficiently large,
        \begin{equation*}
            \label{eq::first_cond}
            f\left(\gamma_k\left(\delta^q\right)\right) > f(x_k) + \sigma\delta^q\nabla f(x_k)^\top d_k.
        \end{equation*}
        Since $f$ and $\gamma_k$ (Assumption \ref{ass::defcurve}) are continuously differentiable, recalling that $\gamma_k(0) = x_k$, by the Mean Value Theorem we know that, for all $k$,
        \begin{equation*}
            \label{eq::mvt}
            f\left(\gamma_k\left(\delta^q\right)\right) = f(x_k) + \delta^q\nabla f(\gamma_k(c_k\delta^q))^\top\gamma_k'(c_k\delta^q),
        \end{equation*}
        with $c_k \in (0, 1)$.
        Combining then the last two equations, we obtain that, for all $k\in\bar{K}$ sufficiently large,
        \begin{equation}
            \label{eq::first_with_mvt}
            \nabla f(\gamma_k(c_k\delta^q))^\top\gamma_k'(c_k\delta^q) > \sigma\nabla f(x_k)^\top d_k.
        \end{equation}
        
        Since $\{\gamma_k\}$ is gradient-related to $\{x_k\}$, Assumption \ref{ass::xi} holds for $\{s_k\}$ and $c_k \in (0, 1)$ for all $k$, the subsequences $\{d_k\}_{\bar{K}}$, $\{s_k\}_{\bar{K}}$ and $\{c_k\}_{\bar{K}}$ are bounded; thus, there exists $\bar{K}_1\subseteq \bar{K}$ such that $d_k \to \bar{d}$, $s_k \to \bar{s}$ and $c_k \to \bar{c} \in [0, 1]$ for $k \in \bar{K_1}$, $k \to \infty$. 
        Since $\gamma_k$ is continuous, we then have
        \begin{equation*}
            \lim_{\substack{k\in \bar{K}_1\\k\to\infty}}\gamma_k(c_k\delta^q) = \lim_{\substack{k\in \bar{K}_1\\k\to\infty}}\gamma(c_k\delta^q;x_k,d_k,s_k) = \gamma(\bar{c}\delta^q;\bar{x},\bar{d},\bar{s}) = \bar{\gamma}(\bar{c}\delta^q).
        \end{equation*}
        Similarly, it follows by the continuous differentiability of $\gamma_k$ that 
        \begin{equation*}
            \lim_{\substack{k\in \bar{K}_1\\k\to\infty}}\gamma'_k(c_k\delta^q) = \gamma'(\bar{c}\delta^q;\bar{x},\bar{d},\bar{s}) = \bar{\gamma}'(\bar{c}\delta^q).
        \end{equation*}
        Given that $f$ is continuously differentiable, we can now take the limit for $k \in \bar{K}_1$, $k \to \infty$ in equation \eqref{eq::first_with_mvt} to obtain 
        \begin{equation*}
            \nabla f(\bar{\gamma}(\bar{c}\delta^q))^\top\bar{\gamma}'(\bar{c}\delta^q) \ge \sigma\nabla f(\bar{x})^\top \bar{d}.    
        \end{equation*}
        The previous inequality holds for any $q\in\mathbb{N}$; thus, we can further take the limit for $q\to\infty$ and get that 
        \begin{equation*}
            \nabla f(\bar{\gamma}(0))^\top \bar{\gamma}'(0) = \nabla f(\bar{x})^\top \bar{d} \geq \sigma \nabla f(\bar{x})^\top\bar{d},
        \end{equation*}
        that is, $(1-\sigma)\nabla f(\bar{x})^\top \bar{d}\geq0$. Given $\sigma \in (0, 1)$, the last result implies $\nabla f(\bar x)^\top \bar{d}\ge0$. Since $\{\gamma_k\}$ is gradient-related to $\{x_k\}$ (Definition~\ref{def::gradient-related}), we finally have that 
        \begin{equation*}
            0 \le \nabla f(\bar x)^\top \bar{d} = \nabla f(\bar x)^\top \bar{\gamma}'(0) \le \limsup_{\substack{k \in \bar{K}_1\\k \to \infty}}\nabla f(x_k)^\top \gamma_k'(0) < 0,
        \end{equation*}
        which is a contradiction.

        \item[(ii)] $\tilde{K}\subseteq K$ where, for all $k \in \tilde{K}$, $\gamma_k$ is a straight line (line \ref{line::restore} of Algorithm \ref{alg::CSM_constrained} is then executed). Since Assumption \ref{ass::dk_feas} holds and $\delta \in (0, 1)$, we have that, for all $k \in \tilde{K}$ sufficiently large, $\gamma_k(\delta^q) = x_k + \delta^qd_k \in \Omega$. Hence, the first condition in \eqref{eq::armijo_unsatisfied} holds and the argument from case (i) applies, leading to a contradiction.
        
        \item[(iii)] $\hat{K}\subseteq K$ where, for all $k \in \hat{K}$, $\gamma_k$ is a curve (line \ref{line::restore} of Algorithm \ref{alg::CSM_constrained} is thus not executed) and, for all $k \in \hat{K}$ sufficiently large, $\gamma_k\left(\delta^q\right) \not \in \Omega$, i.e.,
        \begin{equation*}
            \exists i_k \in [m]: g_{i_k}(\gamma_k(\delta^q)) > 0.
        \end{equation*}
        Since the set $[m]$ is finite, we can consider a subsequence $\hat{K}_1 \subseteq \hat{K}$ such that, for all $k \in \hat{K}_1$ sufficiently large, $i_k = \hat{i}$ and 
        \begin{equation}
            \label{eq::g_before_lim}
            g_{\hat{i}}(\gamma_k(\delta^q)) > 0.
        \end{equation}
        It thus follows from equation \eqref{eq::xkL0}, Assumption \ref{ass::dk_feas}, the last result and Lemma \ref{lemma::CHinfeas} that 
        \begin{equation}
            \label{eq::g_q0}
            g_{\hat{i}}(x_k+s_k) > 0.
        \end{equation}
        By an argument similar to that in case (i), we can consider a subsequence $\hat{K}_2 \subseteq \hat{K}_1$ such that $d_k \to \hat{d}$, $s_k \to \hat{s}$ for $k \in \hat{K}_2$, $k \to \infty$, and
        \begin{equation*}
            \begin{gathered}
                \lim_{\substack{k\in \hat{K}_2\\k\to\infty}}\gamma_k(\delta^q) = \lim_{\substack{k\in \hat{K}_2\\k\to\infty}}\gamma(\delta^q;x_k,d_k,s_k) = \gamma(\delta^q;\bar{x},\hat{d},\hat{s}) = \hat{\gamma}(\delta^q),\\
                \lim_{\substack{k\in \hat{K}_2\\k\to\infty}}\gamma'_k(\delta^q) = \gamma'(\delta^q;\bar{x},\hat{d},\hat{s}) = \hat{\gamma}'(\delta^q).
            \end{gathered}
        \end{equation*}
        Since $g$ is continuous, we can take the limit for $k \in \hat{K}_2$, $k \to \infty$, in equation \eqref{eq::g_before_lim} to obtain
        \begin{equation}
            \label{eq::g_forallq}
            g_{\hat{i}}(\hat{\gamma}(\delta^q)) \ge 0.
        \end{equation}
        The previous inequality holds for any $q \in \mathbb{N}$; thus, we can further take the limit for $q \to \infty$ and get $g_{\hat{i}}(\hat{\gamma}(0)) \ge 0$. Since, by equation \eqref{eq::xkL0}, $\hat{\gamma}(0) = \bar{x} \in \mathcal{L}_0 \subseteq \Omega$, it must hold that 
        \begin{equation}
            \label{eq::gbarx}
            g_{\hat{i}}(\hat{\gamma}(0)) = g_{\hat{i}}(\bar{x}) = 0.
        \end{equation}

        Let us now define the composite function $h_{\hat{i}}: \mathbb{R}\to\mathbb{R}$ as $h_{\hat{i}}(t)=g_{\hat{i}}(\hat{\gamma}(t))$; since $g_{\hat{i}}$ and $\hat{\gamma}$ (Assumption \ref{ass::defcurve}) are continuously differentiable, $h_{\hat{i}}$ is also continuously differentiable. By equations \eqref{eq::g_forallq}-\eqref{eq::gbarx}, we then deduce that
        \begin{equation}
            \label{eq::g_bari_inequality}
            0 \le \lim_{q\to\infty}\frac{h_{\hat{i}}(\delta^q) - h_{\hat{i}}(0)}{\delta^q} = h_{\hat{i}}'(0) = \nabla g_{\hat{i}}(\hat{\gamma}(0))^\top\hat{\gamma}'(0) = \nabla g_{\hat{i}}(\bar{x})^\top\hat{d}.
        \end{equation}
        By convexity and continuous differentiability of $g$, and equations \eqref{eq::gbarx}-\eqref{eq::g_bari_inequality}, we thus obtain that 
        \begin{equation*}
            g_{\hat{i}}(\bar{x} + \tilde{t}\hat{d}) \ge g_{\hat{i}}(\bar{x}) + \tilde{t}\nabla g_{\hat{i}}(\bar{x})^\top \hat{d} \ge 0,
        \end{equation*}
        with $\tilde{t} \in (0, 1)$ defined as in Algorithm \ref{alg::CSM_constrained}.
        
        Since $g_{\hat{i}}$ is continuous and $\{\varepsilon_k\} \subseteq \mathbb{R}_+$ is a decreasing sequence in Algorithm~\ref{alg::CSM_constrained}, it follows that, for all $k \in \hat{K}_2$ sufficiently large, $g_{\hat{i}}(x_k + \tilde{t} d_k) \ge -\varepsilon_k$, which implies that $\hat{i} \in \tilde{I}_k$. Recalling that equation~\eqref{eq::g_q0} holds for all $k \in \hat{K}_1$ sufficiently large, the first condition in the \emph{if} clause of line~\ref{line::if_restore} is then satisfied for all $k \in \hat{K}_2$ sufficiently large. Consequently, line~\ref{line::restore} would be executed. This leads to a contradiction since, by definition of the subsequence $\hat{K}$, line~\ref{line::restore} is not executed at any iteration $k \in \hat{K}$.
    \end{enumerate}
    
    We obtain a contradiction in all the three cases. Thus, the proof is complete.
\end{proof}

\section{Further Developments of the Curve Search Method for Constrained Optimization}
\label{sec::extensions}

In this section, we discuss the integration of additional techniques into the curve search method proposed in Algorithm~\ref{alg::CSM_constrained}, one at a time. Note that for each of these techniques we also describe the modifications required in Algorithm~\ref{alg::CSM_constrained} in order to incorporate them. For completeness, the full scheme of the method with all these updates included is reported in Appendix~\ref{app::extended_scheme}.

\subsection{Non-monotone Curve Search}

Similarly to what is done in \cite{donnini2025efficientglobalizationheavyballtype}, Algorithm \ref{alg::CSM_constrained} can integrate a non-monotone version of the Armijo-type curve search used in line \ref{line::armijo}. In the non-monotone version we require that the stepsize $t_k$ must satisfy the following two conditions:
\begin{equation}\label{eq::nm-fcAR}
    \gamma_k(t_k)\in\Omega \quad \text{ and } \quad f(\gamma_k(t_k))\leq\max_{0\leq j\leq m(k)} f(x_{k-j})+\sigma t_k\nabla f(x_k)^\top d_k,
\end{equation}
where $m(0)=0$, $0\leq m(k)\leq\min\{m(k-1)+1,M\}$ for $k\geq1$ and $M\in\mathbb{N}$. Note that, even if we modify the sufficient decrease condition, we still require the point $\gamma_k(t_k)$ to be feasible at each iteration.

The finite termination (Proposition \ref{prop::ACSprop}) and convergence (Proposition \ref{prop:globalconvm}) results can be adapted as follows taking into account the use of the non-monotone curve search.

\begin{proposition}
    Let $x_k \in \Omega$ and $\gamma_k$ be a gradient-related curve of the form \eqref{eq::quadratic_gamma} satisfying Assumptions \ref{ass::defcurve}-\ref{ass::dk_feas}. Then, the non-monotone Armijo-type curve search procedure in equation \eqref{eq::nm-fcAR} terminates in a finite number of iterations returning a valid stepsize $t_k > 0$. Moreover, one of the two following conditions holds:
    \begin{enumerate}
        \item[(i)] $t_k=1$;
        \item[(ii)] if $t_k<1$, then $\gamma_k(\frac{t_k}{\delta}) \not \in \Omega$ or $f(\gamma_k(\frac{t_k}{\delta}))>\max_{0\leq j\leq m(k)} f(x_{k-j})+\sigma \frac{t_k}{\delta}\nabla f(x_k)^\top d_k$.
    \end{enumerate}
\end{proposition}
\begin{proof}
    The result follows as in Proposition \ref{prop::ACSprop}, using \cite[Section 20.2]{grippo2023introduction} together with \cite[Section 24.4]{grippo2023introduction}, and \cite[Proposition 3]{donnini2025efficientglobalizationheavyballtype} instead of \cite[Proposition 1]{donnini2025efficientglobalizationheavyballtype}.
\end{proof}

\begin{proposition}
    Let Assumption \ref{ass::L0} hold. Let $\{(x_k,\gamma_k)\}$ be the sequence generated by Algorithm \ref{alg::CSM_constrained} such that all the hypotheses of Proposition \ref{prop:globalconvm} hold and, for all $k$, stepsizes $t_k$ are chosen by the non-monotone Armijo-type curve search in equation \eqref{eq::nm-fcAR}.
    Then, the sequence $\{x_k\}$ admits accumulation points and there exists at least one accumulation point $\bar{x}$ which is stationary for problem \eqref{eq::con-min}.
\end{proposition}
\begin{proof}
    The thesis can be proven by reasoning similar to that used in the proof of \cite[Proposition 4]{donnini2025efficientglobalizationheavyballtype}, where Proposition \ref{prop:globalconvm} is used in place of Proposition 2 from the referenced paper.
\end{proof}

In \cite{donnini2025efficientglobalizationheavyballtype}, a stronger convergence result, ensuring that every accumulation point of the generated sequence is stationary, is established. This result required the introduction of an additional assumption on the sequence of curves generated by the method. We did not further investigate whether this result can also be extended to the constrained case considered here, as we believe that such an analysis lies somewhat outside the scope of the present paper. Nevertheless, we are confident that a stronger convergence result in the non-monotone case, similar to the one obtained in the monotone case, can be achieved, although not in a completely straightforward way. Such an investigation may represent a possible direction for future work.

\subsection{Spectral Curve Search Method}

Motivated by the state-of-the-art performance of the \texttt{SPG} method introduced in \cite{birgin2000nonmonotone} and briefly described in Section \ref{sec::preliminaries}, we propose the incorporation of a spectral gradient strategy into the definition of the curve $\gamma_k$ with the aim of enhancing the performance of our curve search method.

In particular, whenever the direction $d_k$ is computed as in \texttt{SPG} (see equation \eqref{eq::d_SPG}), we propose a spectral rescaling of the momentum term in the second direction $s_k$:
\begin{equation*}
    s_k=\alpha d_k+\beta\eta_k(x_k-x_{k-1}).
\end{equation*}

Note that, when $d_k$ is the \texttt{SPG} direction, this rescaling of the momentum term $\beta(x_k-x_{k-1})$ by the ``inverse Rayleigh quotient'' $\eta_k$ also helps prevent the directions $d_k$ and $s_k$ from having significantly different magnitudes.  This would lead to \emph{degenerate} quadratic curves
\eqref{eq::quadratic_gamma} which collapse onto one of the two control directions, $d_k$ or $s_k$.

Note that, as already mentioned in Section \ref{sec::preliminaries}, the direction used in \texttt{SPG} is a projected gradient–type direction and is therefore gradient-related. It is easy to verify that all the other assumptions used to establish the theoretical results for the proposed curve search method remain valid also with these definitions of $d_k$ and $s_k$.

\subsection{Adaptive Momentum Strategy}
\label{subsec::adpative_momentum_strategy}

Degenerate quadratic curves may arise when the computation of projected gradient-type direction $d_k$ requires an actual projection. In such cases, the momentum term can make $s_k$ significantly larger in magnitude than $d_k$. This is particularly likely to happen as the iterates approach the boundary of the feasible set.

To mitigate this issue, we propose an \textit{adaptive momentum} strategy, following a similar reasoning as in \cite{fan2023msl}, in which the parameter $\beta$ is iteratively reduced until the resulting direction $s_k$ becomes feasible. The procedure is outlined in Algorithm \ref{alg::lines7-8}, which shows how lines \ref{line::if_restore}-\ref{line::restore} of Algorithm \ref{alg::CSM_constrained} should be modified to incorporate this strategy.

\begin{algorithm}[ht]\caption{Lines \ref{line::if_restore}-\ref{line::restore} of Algorithm \ref{alg::CSM_constrained} for adaptive momentum strategy}
    \label{alg::lines7-8}
	\begin{algorithmic}[1]
        \IF{$\exists i\in \tilde{I}_k: g_i(x_k+s_k) > 0$ \textbf{or} $k = 0$\label{line::if-restore-2}}
        \STATE Set $s_k = d_k$ 
        \ELSE
        \IF{projection was required to define $d_k$}
        \STATE Find $\beta_k = \max_{h \in \mathbb{N}}\{\delta^h\beta \mid x_k + \alpha d_k + \delta^h\beta(x_k - x_{k-1}) \in \Omega\}$\label{line::reduce_beta}
        \STATE Set $s_k = \alpha d_k + \beta_k(x_k - x_{k-1})$
        \ENDIF
        \ENDIF
	\end{algorithmic}
\end{algorithm}

Note that in line \ref{line::reduce_beta} of Algorithm \ref{alg::lines7-8} we have used the same decay factor $\delta$ used for the stepsize $t_k$ in line \ref{line::armijo} of Algorithm \ref{alg::CSM_constrained}; of course, a different choice for the decay factor could have been employed. 

In the following lemma, we state that, under a reasonable assumption on the choice of the momentum parameter $\alpha$, the adaptive reduction of $\beta$ terminates in a finite number of steps and, consequently, the resulting overall Algorithm \ref{alg::CSM_constrained} remains well-defined. The proof of the lemma is provided in Appendix \ref{app::proof}.

\begin{lemma}
    \label{lemma::reduction_beta}
    Let $k > 0$ be an iteration of Algorithm \ref{alg::CSM_constrained} where $\alpha \in (0, 1)$, $\gamma_k$ is a gradient-related curve satisfying Assumption \ref{ass::dk_feas} and projection was required to define $d_k$. Then, line \ref{line::reduce_beta} of Algorithm \ref{alg::lines7-8} is well-defined, i.e., it terminates after a finite number of iterations returning a valid $\beta_k > 0$. 
\end{lemma}

\section{Computational Experiments}
\label{sec::experiments}

In this section, we report the results of numerical experiments in order to validate the goodness and consistency of our approach in smooth convexly constrained optimization contexts. The code was written in \texttt{Python3}\footnote{The implementation of the proposed approach is available at \url{https://github.com/pierlumanzu/constrained_scs}.}. All experiments were conducted on a computer with the following specifications: Ubuntu 24.04 OS, Intel(R) Xeon(R) Gold 6330N CPU @ 2.20GHz and 125 GB of RAM.

\subsection{Experimental Settings}
\label{subsec::settings}
To the best of our knowledge,  the spectral projected gradient method \texttt{SPG}, introduced in \cite{birgin2000nonmonotone, birgin01} and briefly described in Section~\ref{sec::preliminaries}, represents the state-of-the-art for solving smooth convexly constrained optimization problems in which the Euclidean projection is computationally tractable. For this reason, we proposed a comparison between \texttt{SPG} and our approach outlined in Algorithm~\ref{alg::CSM_constrained}, with $d_k$ being the \texttt{SPG} direction (see equation \eqref{eq::d_SPG}) and all the techniques presented in Section~\ref{sec::extensions} being incorporated. In the remainder of the section, we refer to our approach as \texttt{SCS}. 

For both approaches, we considered the monotone and non-monotone variants, with the latter using parameter $M = 10$, as in \cite{birgin2000nonmonotone} for testing \texttt{SPG}. Based on preliminary tuning experiments, not reported here for brevity, we set the algorithmic parameters as follows: $\delta = 0.5$, $\sigma = 10^{-7}$, $\alpha = 0.999$, $\eta_{\min} = 10^{-3}$, $\eta_{\max} = 10^3$, $\varepsilon_0=0.1$ and $\varepsilon_{k} = 0.95 \varepsilon_{k-1}$ for $k > 0$. As for the remaining parameters of \texttt{SPG}, including those related to the quadratic interpolation-based line search, they were set according to \cite{birgin2000nonmonotone, birgin01}. Regarding the adaptive momentum strategy (Section~\ref{subsec::adpative_momentum_strategy}), based on preliminary experiments, we adopted a dynamic update of the momentum parameter $\beta$. Specifically, starting from $\beta = 0.9$, we updated it depending on whether the adaptive momentum step was performed: if it was executed, then $\beta$ was set for the next iteration equal to the $\beta_k$ selected by the adaptive strategy (see Algorithm \ref{alg::lines7-8}); otherwise, we set for the next iteration $\beta = \min\{0.9, \frac{\beta}{\delta}\}$. Note that this update rule does not affect the convergence theory. We run the algorithms until one of the following stopping criteria was met: the current solution satisfied the stationarity condition with an approximation degree $\epsilon = 10^{-3}$; the number of iterations reached $5000$; a time limit of $2$ minutes was exceeded. If an algorithm met one of the last two criteria without satisfying the stationarity condition, it was considered a failure.

We conducted our experiments on constrained versions of the benchmark problems listed in Table~\ref{tab::problems}, drawn from the \texttt{CUTEst} collection \cite{gould2015cutest}. The problem dimension $n$ ranges from 2 to 5625.
\begin{table}[htb]
    \tbl{Benchmark of \texttt{CUTEst} problems.}
    {\begin{tabular}{lc} \toprule
    $n$ & \textit{Problem} \\ \midrule
    2 & JENSMP \\ \midrule
    3 & BARD, CHWIRUT1LS, ENGVAL2 \\ \midrule
    4 & ROSZMAN1LS \\ \midrule
    5 & MGH17LS \\ \midrule
    6 & FBRAIN3LS, HEART6LS \\ \midrule
    7 & CERI651ALS, CERI651BLS, CERI651ELS, HAHN1LS, PALMER1D \\ \midrule
    8 & GAUSS1LS, GAUSS3LS, HEART8LS, PALMER1C, PALMER4C, PALMER8C \\ \midrule
    11 & OSBORNEB \\ \midrule
    50 & CHNROSNB, ERRINRSM \\ \midrule
    66 & DIAMON2DLS, DMN15102LS, DMN37142LS \\ \midrule
    98 & LUKSAN12LS \\ \midrule
    99 & DIAMON3DLS, DMN15103LS, DMN15333LS, DMN37143LS \\ \midrule
    100 & LUKSAN11LS, LUKSAN22LS \\ \midrule
    102 & SPIN2LS \\ \midrule
    134 & COATING \\ \midrule
    200 & ARGLINA, ARGLINB, ARGTRIGLS \\ \midrule
    1000 & EXTROSNB, FLETCHCR, KSSLS \\ \midrule
    2550 & EIGENALS, EIGENBLS \\ \midrule
    2652 & EIGENCLS \\ \midrule
    4000 & CHAINWOO \\ \midrule
    5000 & CRAGGLVY, FLETCBV2, GENHUMPS, TESTQUAD, TRIDIA, VAREIGVL \\ \midrule
    5625 & FMINSURF \\ \bottomrule
    \end{tabular}}
    \label{tab::problems}
\end{table}
Each problem was tested with four different constraint configurations:

\begin{itemize}
    \item feasible set defined by a hyper-sphere constraint
    \begin{equation*}
        \Omega_{\text{Sph}} = \{x \in \mathbb{R}^n \mid \|x\|^2 - 100 \le 0\};
    \end{equation*}

    \item feasible set defined by a hyper-ellipsoid constraint
    \begin{equation*}
        \Omega_{\text{Ell}} = \{x \in \mathbb{R}^n \mid (x - c)^\top P^{-1} (x-c) - 25 \le 0\},
    \end{equation*}
    where $c = [1,\ldots,1]^\top$ and $P$ is a randomly drawn positive definite diagonal matrix;

    \item feasible set defined by hyper-sphere, hyper-plane and box constraints
    \begin{equation*}
        \begin{aligned}
            \Omega_{\text{Com}} & = \left\{x \in \mathbb{R}^n
            \;\middle\vert\;
            \begin{aligned}
                & \|x - c\|^2 - 100 \le 0, \\
                & w^\top x - 5 \le 0, \\
                & -5 \le x^i \le 10, \quad \forall i \in \{1,\ldots,n\}
            \end{aligned}
            \right\},
        \end{aligned}
    \end{equation*}
    where $c = [4,\ldots,4]^\top$ and $w = [\frac{1}{n}, \ldots, \frac{1}{n}]^\top$;

    \item feasible set with box constraints
    \begin{equation*}
        \Omega_{\text{Box}} = \{x \in \mathbb{R}^n \mid -1 \le x^i \le 1, \ \forall i \in \{1,\ldots,n\}\}.
    \end{equation*}
\end{itemize}
The projection onto $\Omega_{\text{Sph}}$ and $\Omega_{\text{Box}}$ can be computed efficiently in closed form, namely: $\Pi_{\Omega_{\text{Sph}}}[x] = \frac{10 x}{\|x\|}$ and $\Pi_{\Omega_{\text{Box}}}[x] = \hat{x}$ with $\hat{x}^i = \min\{1, \max\{-1, x^i\}\}$ for all $i \in \{1,\ldots,n\}$. For other constraint configurations, the auxiliary Euclidean projection problem is solved using the commercial solver Gurobi (version 12) \cite{gurobi}. For each problem, all algorithms were initialized with the same starting point provided by the \texttt{CUTEst} library and projected onto the considered feasible set.

To evaluate the performance of the algorithms, we used two metrics: the final objective function value $f^\star$ achieved by each method and the corresponding execution time $T$. For a concise visualization of the results, we employed performance profiles \cite{Dolan2002}. When computing these profiles, we included only the problem instances for which at least one algorithm successfully returned an (approximate) stationary point.

\subsection{Performance Overview}

In Figure \ref{fig::all}, we show the performance profiles obtained by the non-monotone and monotone versions of \texttt{SCS} and \texttt{SPG} on the problems listed in Table \ref{tab::problems}. We analyzed the performance of the algorithms for each of the tested constraint configurations described in Section \ref{subsec::settings}.

Regarding problem instances with hyper-spherical constraint ($\Omega_{\text{Sph}}$), \texttt{SCS} ($M=10$) and the monotone version of \texttt{SPG} exhibited very similar efficiency, and both significantly outperformed the standard non-monotone version of \texttt{SPG}. The monotone version of \texttt{SCS} was observed to be slightly less efficient than \texttt{SCS} ($M=10$) and monotone \texttt{SPG}; however, in terms of robustness, it ranked as the second best, surpassed only by its non-monotone counterpart. For problems with hyper-ellipsoidal constraint ($\Omega_{\text{Ell}}$) or multiple constraints ($\Omega_{\text{Com}}$), \texttt{SCS} ($M=10$) and \texttt{SCS} outperformed \texttt{SPG} ($M=10$) and \texttt{SPG}, respectively. In these settings, our method in its non-monotone version clearly emerged as the best in terms of computational time $T$. Finally, for problems with box constraints ($\Omega_{\text{Box}}$), the situation is less clear than in the previous cases: \texttt{SPG} algorithms appeared to be slightly superior, although our methods remained competitive overall.

\begin{figure}[!h]
    \centering
    \subfloat[16 problem instances with feasible set $\Omega_{\text{Sph}}$.]{\includegraphics[width=0.35\textwidth]{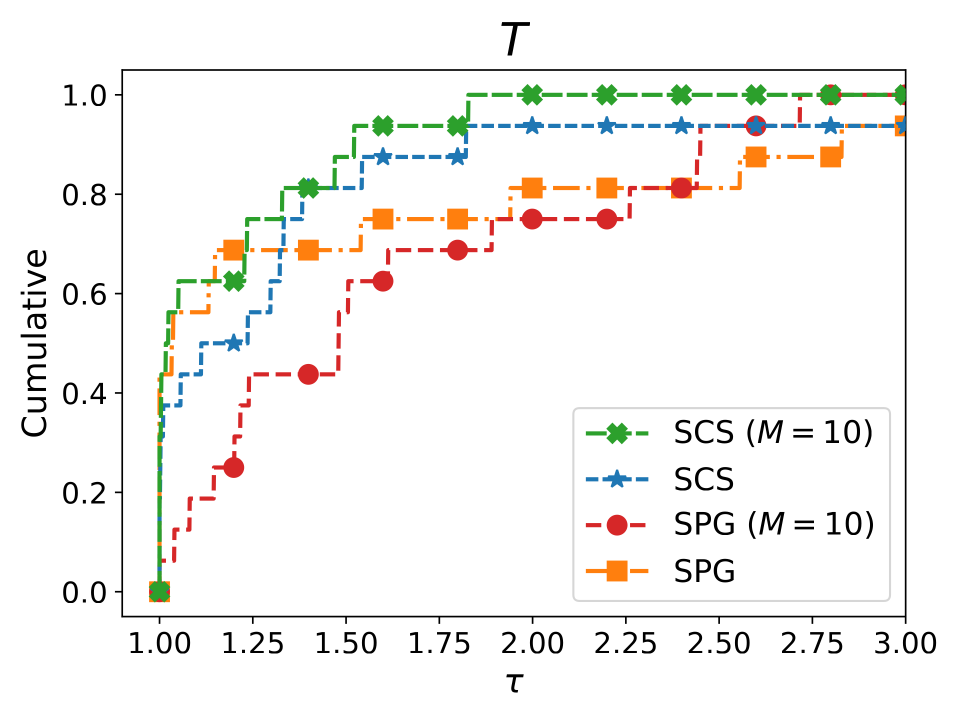}}
    \hfil
    \subfloat[36 problem instances with feasible set $\Omega_{\text{Ell}}$.]{\includegraphics[width=0.35\textwidth]{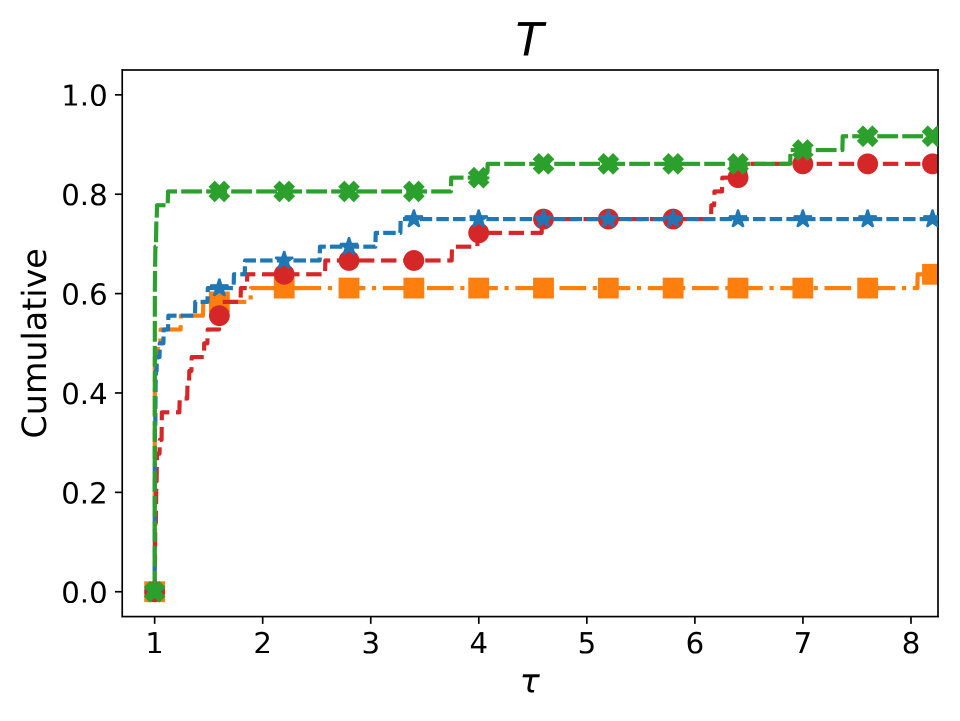}}
    \\
    \subfloat[50 problem instances with feasible set $\Omega_{\text{Com}}$.]{\includegraphics[width=0.35\textwidth]{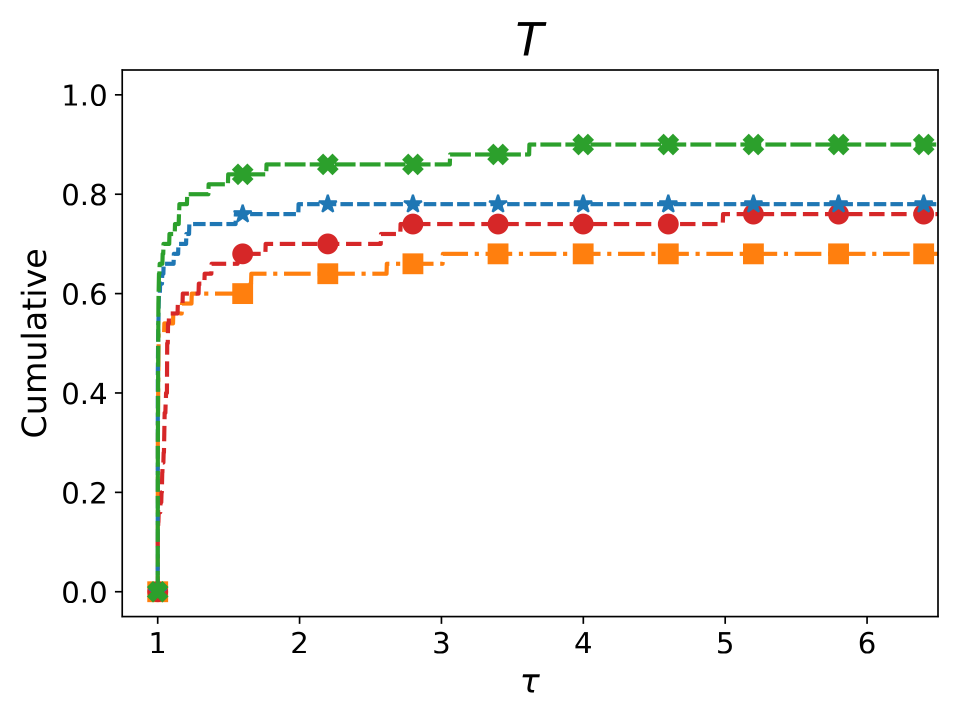}}
    \hfil
    \subfloat[35 problem instances with feasible set $\Omega_{\text{Box}}$.]{\includegraphics[width=0.35\textwidth]{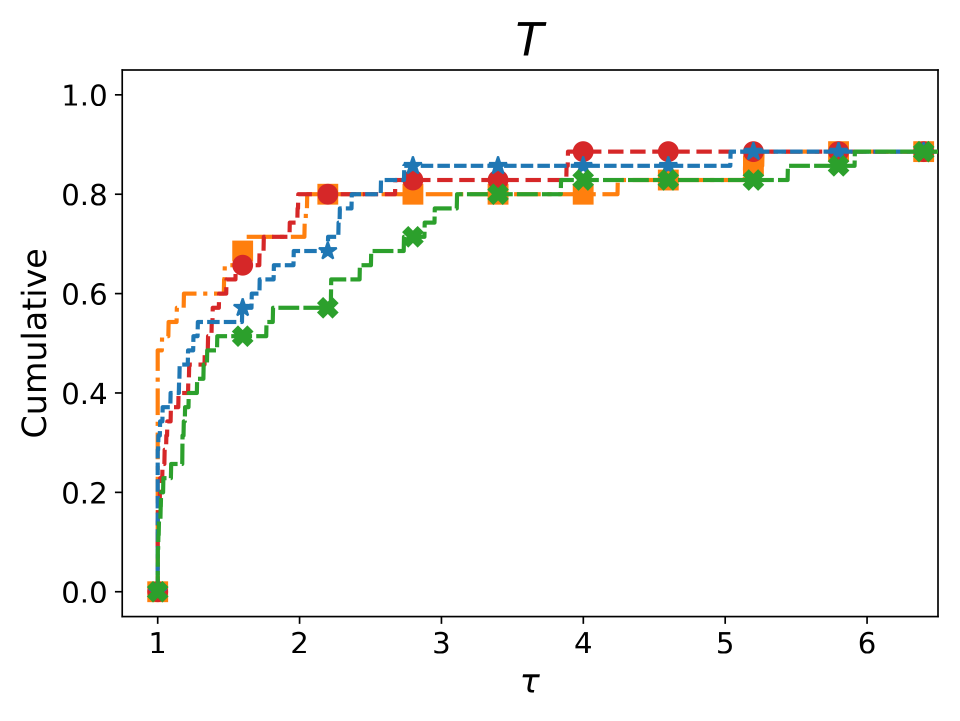}}
    \caption{Performance profiles obtained by \texttt{SCS} and \texttt{SPG} for both the non-monotone version with $M=10$ and the monotone one with respect to $T$ on the problems of Table \ref{tab::problems}. Note that the intervals of the x-axes were set to improve the visualization of the numerical results. For further details on the computation of the performance profiles, we refer the reader to Section~\ref{sec::experiments}.}
    \label{fig::all}
\end{figure}

Since we are dealing with a constrained setting, we found it of interest to perform the same analysis by considering only those problems whose solutions lie on the boundary of the feasible set. In Figure \ref{fig::active}, we therefore report the performance profiles obtained by the non-monotone and monotone versions of \texttt{SCS} and \texttt{SPG} on this subset of problems, again analyzing performance with respect to the type of constraints considered. As can be observed, conclusions very similar to the previous ones can be drawn. The only substantial difference is that the monotone version of \texttt{SCS} appears to be more robust in terms of $T$ than its non-monotone counterpart on the considered problems with hyper-ellipsoid constraint ($\Omega_{\text{Ell}}$).

\begin{figure}[!h]
    \centering
    \subfloat[14 problem instances with feasible set $\Omega_{\text{Sph}}$.]{\includegraphics[width=0.35\textwidth]{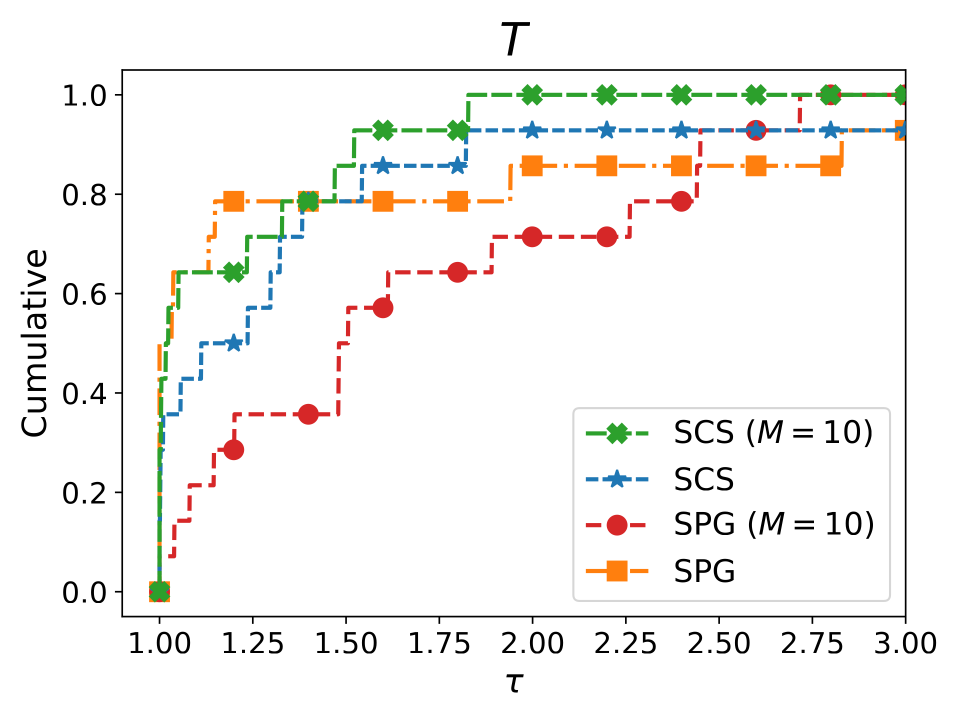}}
    \hfil
    \subfloat[15 problem instances with feasible set $\Omega_{\text{Ell}}$.]{\includegraphics[width=0.35\textwidth]{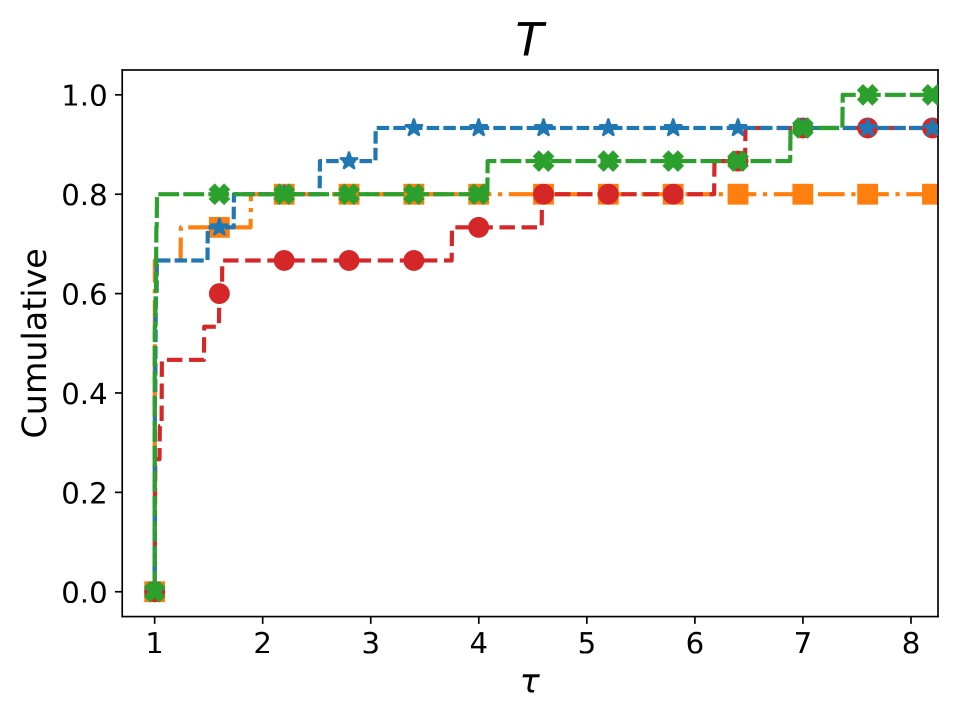}}
    \\
    \subfloat[12 problem instances with feasible set $\Omega_{\text{Com}}$.]{\includegraphics[width=0.35\textwidth]{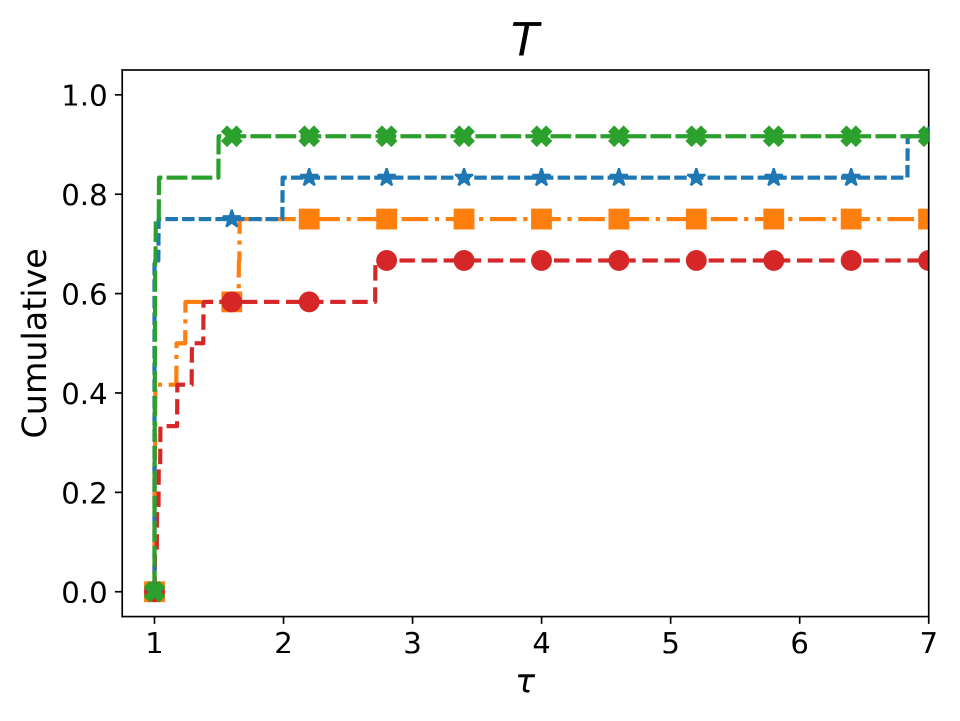}}
    \hfil
    \subfloat[27 problem instances with feasible set $\Omega_{\text{Box}}$.]{\includegraphics[width=0.35\textwidth]{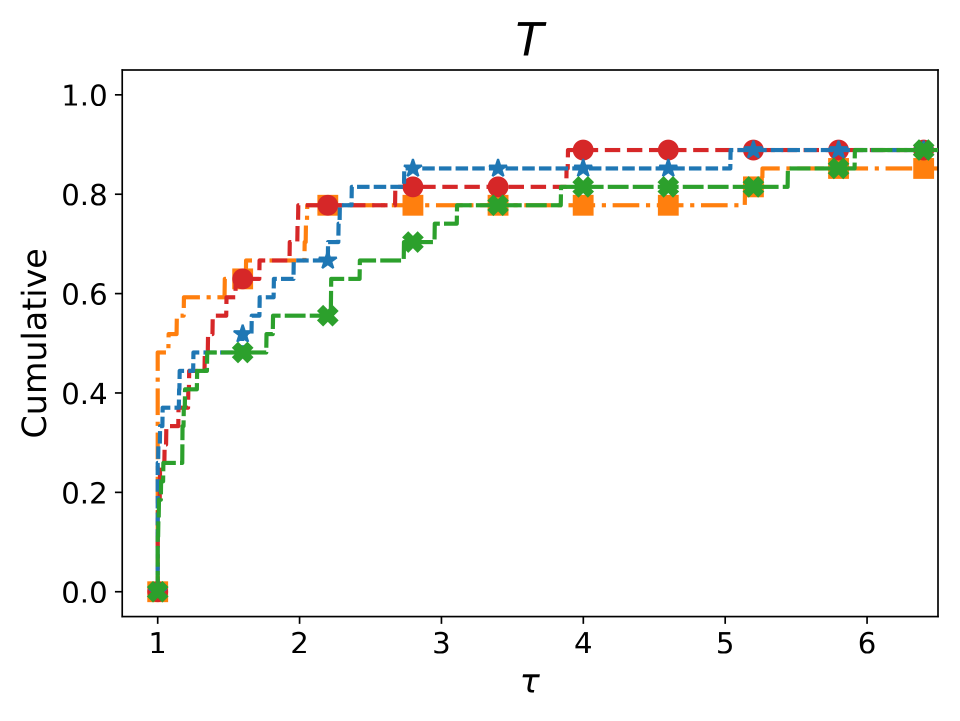}}
    \caption{Performance profiles obtained by \texttt{SCS} and \texttt{SPG} for both the non-monotone version with $M=10$ and the monotone one with respect to $T$ on the problems of Table~\ref{tab::problems} whose solution lies on the boundary of the feasible set. Note that the intervals of the x-axes were set to improve the visualization of the numerical results. For further details on the computation of the performance profiles, we refer the reader to Section~\ref{sec::experiments}.}
    \label{fig::active}
\end{figure}

Finally, it is noteworthy from Figure \ref{fig::all_obj} that the efficiency improvements demonstrated by our methods are related to overall good quality stationary solutions, both across the full problem benchmark and for the subset of problems with solutions on the boundary of the feasible set.

\begin{figure}[!h]
    \centering
    \subfloat[137 problem instances.]{\includegraphics[width=0.35\textwidth]{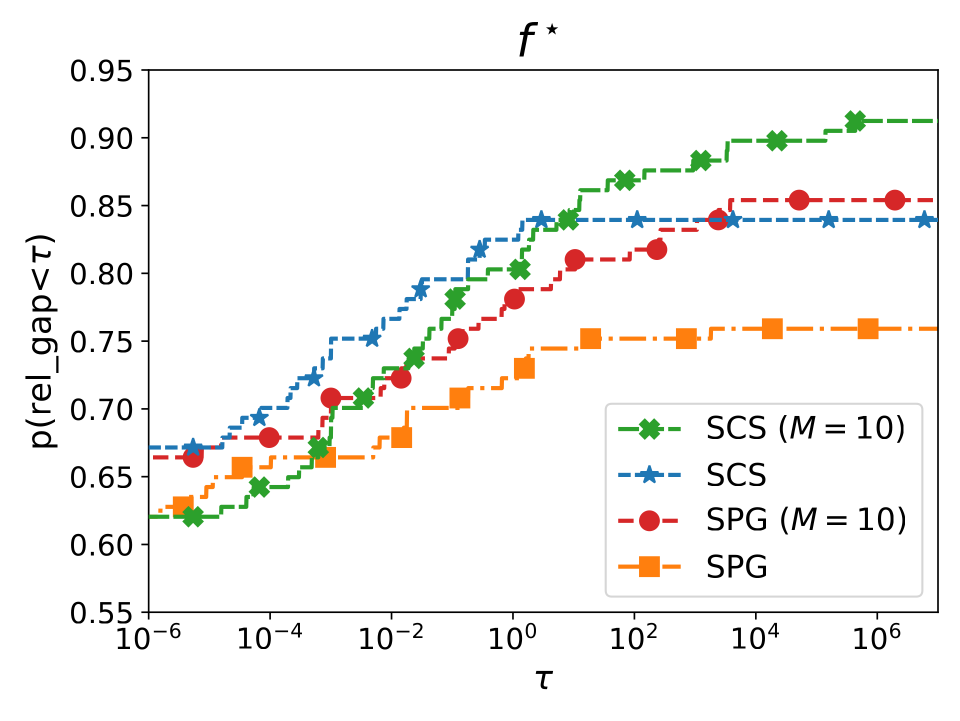}}
    \hfil
    \subfloat[68 problem instances whose solution lies on the boundary of the feasible set.]{\includegraphics[width=0.35\textwidth]{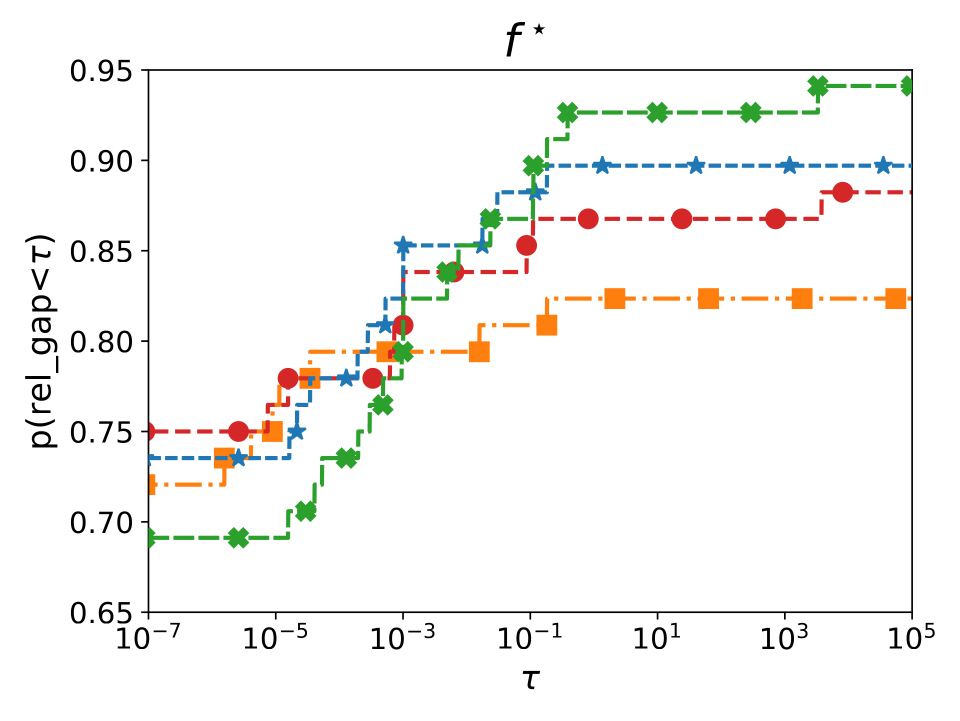}}
    \caption{Performance profiles obtained by \texttt{SCS} and \texttt{SPG} for both the non-monotone version with $M=10$ and the monotone one with respect to $f^\star$ on the problems of Table~\ref{tab::problems} with all the tested constraint types. Note that the intervals of the axes were set to improve the visualization of the numerical results. For further details on the computation of the performance profiles, we refer the reader to Section~\ref{sec::experiments}.}
    \label{fig::all_obj}
\end{figure}

\section{Conclusions}
\label{sec::conclusions}

In this work, we propose a momentum-based method for smooth, convexly constrained optimization problems, for which the Euclidean projection is computationally tractable, extending a curve search globalization method for heavy-ball type algorithms recently introduced in the literature. 

The method combines momentum acceleration with feasibility preservation, automatically reverting to a feasible descent direction when the heavy-ball step is unacceptable. We proved that the algorithm is well-defined and globally convergent to stationary points, despite the challenges introduced by using a heavy-ball type direction in a constrained setting. We also discussed how several practical mechanisms, such as non-monotone curve search, spectral gradient choice of directions defining the curve and an adaptive momentum strategy, can be incorporated into the base algorithm. 

Numerical experiments on a benchmark of constrained problems demonstrate the robustness of the approach and its competitive performance with respect to the state-of-the-art \texttt{SPG} algorithm. Overall, these results suggest that combining curve search globalization with momentum-based strategies provides a promising approach for constrained optimization. Future work may include extending the approach to more general, possibly nonconvex feasible sets.


\section*{Declarations}

\subsection*{Disclosure statement}

The authors report there are no competing interests to declare.

\subsection*{Funding}

No funding was received for conducting this study.

\subsection*{Code availability statement}
The implementation code of the approach presented in the paper can be found at \url{https://github.com/pierlumanzu/constrained_scs}.

\subsection*{Notes on contributors}
\textbf{Federica Donnini} received her bachelor's degree in Mathematics from the University of Florence in 2022. In 2024, she completed a double-degree master's program, obtaining a master's degree in Applied Mathematics from the University of Florence and a master's degree in Mathematical Engineering from Complutense University of Madrid. She is currently a Ph.D. student in the Department of Information Engineering of the University of Florence. Her main research interests include stochastic and robust optimization.
\\\\
\textbf{Pierluigi Mansueto} received his bachelor and master's degree in Computer Engineering at the University of Florence in 2017 and 2020, respectively. He then obtained his PhD degree in Information Engineering from the University of Florence in 2024. Currently, he is a Postdoctoral Research Fellow at the Department of Information Engineering of the University of Florence. His main research interests are multi-objective optimization and global optimization.

\subsection*{ORCID}

Federica Donnini: 0009-0004-0377-2046\\
Pierluigi Mansueto: 0000-0002-1394-0937

\bibliographystyle{tfs}

\begin{thebibliography}{10}
	\providecommand{\MR}{\relax\unskip\space MR }
	\providecommand{\url}[1]{\normalfont{#1}}
	\providecommand{\urlprefix}{Available at }
	
	\bibitem{ben1990curved}
	A. Ben-Tal, A. Melman, and J. Zowe, \emph{Curved search methods for
		unconstrained optimization}, Optimization 21 (1990), pp. 669--695.
	
	\bibitem{bertsekas1999nonlinear}
	D.P. Bertsekas, \emph{Nonlinear programming}, 2nd ed., Athena Scientific,
	Belmont, MA, 1999.
	
	\bibitem{birgin2000nonmonotone}
	E.G. Birgin, J.M. Mart{\'\i}nez, and M. Raydan, \emph{Nonmonotone spectral
		projected gradient methods on convex sets}, SIAM Journal on Optimization 10
	(2000), pp. 1196--1211.
	
	\bibitem{birgin01}
	E.G. Birgin, J.M. Mart\'{\i}nez, and M. Raydan, \emph{Algorithm 813:
		Spg—software for convex-constrained optimization}, ACM Trans. Math. Softw.
	27 (2001), p. 340–349.
	
	\bibitem{botsaris1978differential}
	C.A. Botsaris, \emph{Differential gradient methods}, Journal of Mathematical
	Analysis and Applications 63 (1978), pp. 177--198.
	
	\bibitem{bottou18}
	L. Bottou, F.E. Curtis, and J. Nocedal, \emph{Optimization methods for
		large-scale machine learning}, SIAM Review 60 (2018), pp. 223--311.
	
	\bibitem{Dolan2002}
	E.D. Dolan and J.J. Mor{\'e}, \emph{Benchmarking optimization software with
		performance profiles}, Mathematical Programming 91 (2002), pp. 201--213.
	
	\bibitem{donnini2025efficientglobalizationheavyballtype}
	F. Donnini, M. Lapucci, and P. Mansueto, \emph{Efficient globalization of
		heavy-ball type methods for unconstrained optimization based on curve
		searches} (2025).
	
	\bibitem{fan2023msl}
	C. Fan, S. Vaswani, C. Thrampoulidis, and M. Schmidt, \emph{{MSL}: An Adaptive
		Momentum-based Stochastic Line-search Framework}, in \emph{OPT 2023:
		Optimization for Machine Learning}. 2023.
	
	\bibitem{farin2000essentials}
	G. Farin and D. Hansford, \emph{The essentials of CAGD}, AK Peters/CRC Press,
	2000.
	
	\bibitem{goldfarb1980curvilinear}
	D. Goldfarb, \emph{Curvilinear path steplength algorithms for minimization
		which use directions of negative curvature}, Mathematical programming 18
	(1980), pp. 31--40.
	
	\bibitem{gould2000exploiting}
	N.I. Gould, S. Lucidi, M. Roma, and P.L. Toint, \emph{Exploiting negative
		curvature directions in linesearch methods for unconstrained optimization},
	Optimization methods and software 14 (2000), pp. 75--98.
	
	\bibitem{gould2015cutest}
	N.I. Gould, D. Orban, and P.L. Toint, \emph{{CUTE}st: a constrained and
		unconstrained testing environment with safe threads for mathematical
		optimization}, Computational Optimization and Applications 60 (2015), pp.
	545--557.
	
	\bibitem{grippo2023introduction}
	L. Grippo and M. Sciandrone, \emph{Introduction to methods for nonlinear
		optimization}, Vol. 152, Springer Nature, 2023.
	
	\bibitem{gurobi}
	{Gurobi Optimization, LLC}, \emph{{Gurobi Optimizer Reference Manual}} (2026).
	\urlprefix\url{https://www.gurobi.com}.
	
	\bibitem{lapucci2024globallyconvergentgradientmethod}
	M. Lapucci, G. Liuzzi, S. Lucidi, D. Pucci, and M. Sciandrone, \emph{A globally
		convergent gradient method with momentum}, Computational Optimization and
	Applications 93 (2026), pp. 795--820.
	
	\bibitem{Lee17}
	C.P. Lee, P.W. Wang, W. Chen, and C.J. Lin, \emph{Limited-memory
		Common-directions Method for Distributed Optimization and its Application on
		Empirical Risk Minimization}, in \emph{Proceedings of the 2017 SIAM
		International Conference on Data Mining (SDM)}. SIAM, 2017, pp. 732--740.
	
	\bibitem{lessard2016analysis}
	L. Lessard, B. Recht, and A. Packard, \emph{Analysis and design of optimization
		algorithms via integral quadratic constraints}, SIAM Journal on Optimization
	26 (2016), pp. 57--95.
	
	\bibitem{Liu2024}
	Z. Liu, Y. Ni, H. Liu, and W. Sun, \emph{A new subspace minimization conjugate
		gradient method for unconstrained minimization}, Journal of Optimization
	Theory and Applications 200 (2024), pp. 820--851.
	
	\bibitem{polyak1964some}
	B.T. Polyak, \emph{Some methods of speeding up the convergence of iteration
		methods}, Ussr computational mathematics and mathematical physics 4 (1964),
	pp. 1--17.
	
	\bibitem{polyak1987introduction}
	B.T. Polyak, \emph{Introduction to optimization}, New York, Optimization
	Software, 1987.
	
	\bibitem{Powell1977}
	M.J.D. Powell, \emph{Restart procedures for the conjugate gradient method},
	Mathematical Programming 12 (1977), pp. 241--254.
	
	\bibitem{shi2005new}
	Z.J. Shi and J. Shen, \emph{A new descent algorithm with curve search rule},
	Applied mathematics and computation 161 (2005), pp. 753--768.
	
	\bibitem{Tao22}
	W. Tao, G.W. Wu, and Q. Tao, \emph{Momentum acceleration in the individual
		convergence of nonsmooth convex optimization with constraints}, IEEE
	Transactions on Neural Networks and Learning Systems 33 (2022), pp.
	1107--1118.
	
	\bibitem{wright2022optimization}
	S.J. Wright and B. Recht, \emph{Optimization for data analysis}, Cambridge
	University Press, 2022.
	
	\bibitem{xu2016global}
	Z. Xu, Y. Tang, Z.J. Shi, \emph{et~al.}, \emph{Global convergence of curve
		search methods for unconstrained optimization}, Applied Mathematics 7 (2016),
	p. 721.
	
\end{thebibliography}

\appendix

\section{Supplementary Mathematical Proofs}
\label{app::proof}

In this appendix, we report mathematical proofs which did not find space in the main
body of the manuscript.

\subsection*{Proof of Lemma \ref{lemma::CHgamma}}

\begin{proof}
    First, we note that cases where $\hat{t} = 0$ and/or $t = 0$ trivially follow. Then, we will focus on the case where both values are positive.
    
    By equation \eqref{eq::quadratic_bezier}, we get that $\gamma(\hat{t})=(1-\hat{t})^2P_0+2\hat{t}(1-\hat{t})P_1+\hat{t}^2P_2$, which yields $P_2 = \frac{1}{\hat{t^2}}(\gamma(\hat{t})-(1-\hat{t})^2P_0 - 2\hat{t}(1-\hat{t})P_1)$.
    Replacing now $P_2$ in the generic equation \eqref{eq::quadratic_bezier}, we obtain that
    \begin{equation*}
        \gamma(t)=\left((1-t)^2-\frac{t^2}{\hat{t}^2}(1-\hat{t})^2\right)P_0 + \left(2t(1-t)-2\frac{t^2}{\hat{t}}(1-\hat{t})\right)P_1+\frac{t^2}{\hat{t}^2}\gamma(\hat{t}).
    \end{equation*}
    Recalling the definition of convex hull, i.e.,
    \begin{equation}
        \label{eq::convex_hull}
        \text{CH}(P_0,P_1,\gamma(\hat{t}))=\left\{a_0P_0+a_1P_1+a_2\gamma(\hat{t})\mid \sum_{i=0}^2a_i=1, a_i\geq0,\ \forall i\in\{0,1,2\}\right\},
    \end{equation}
    the thesis immediately follows if 
    \begin{equation*}
        a_0 =  (1-t)^2-\frac{t^2}{\hat t^2}(1-\hat t)^2, \qquad 
        a_1 = 2t(1-t)-2\frac{t^2}{\hat t}(1-\hat t), \qquad
        a_2 =  \frac{t^2}{\hat t^2}
    \end{equation*}
    satisfy the properties indicated in equation \eqref{eq::convex_hull} for $t \in (0, \hat{t}]$.
    
    Observe that $a_0\ge0$ if and only if $f_0(t) \ge f_0(\hat{t})$ with $f_0(h) = \frac{(1-h)^2}{h^2}$. Since $f_0$ is a non-increasing function for all $h\in (0,1]$, $t \le \hat{t}$ and $t, \hat{t} \in (0,1]$, we have that $f_0(t) \ge f_0(\hat{t})$. Following a similar reasoning, it can be shown that $a_1 \ge 0$ for all $t \in (0, \hat{t}]$, as $a_1 \ge 0$ if and only if $f_1(t) \ge f_1(\hat{t})$, where $f_1(h) = \frac{1-h}{h}$ is a non-increasing function on the interval $(0, 1]$. The property $a_2 \ge 0$ holds trivially. Finally, it is easy to verify by direct algebraic manipulation that $a_0 + a_1 + a_2 = 1$, which completes the proof.
\end{proof}

\subsection*{Proof of Lemma \ref{lemma::CHinfeas}}

\begin{proof}
    First, if $\hat{t} = 1$ we trivially get the thesis since, by equations \eqref{eq::quadratic_gamma}-\eqref{eq::quadratic_bezier}, $\gamma(1) = x_k + s_k = P_2$.
    Then, let us consider the case $\hat{t} \in (0, 1)$ and assume, by contradiction, that $g_i(P_2) \le 0$. By the properties of Bézier curves \eqref{eq::quadratic_bezier}, there exist $a_0, a_1, a_2 \ge 0$, with $\sum_{i=0}^2a_i = 1$, such that $\gamma(\hat{t}) = a_0P_0 + a_1P_1 + a_2P_2$. By assumption, we have that $g_i$ is convex and $P_0, P_1 \in \Omega$; thus, it immediately follows that 
    \begin{equation*}
        g_i(\gamma(\hat{t})) = g_i(a_0P_0 + a_1P_1 + a_2P_2) \le a_0g_i(P_0) + a_1g_i(P_1) + a_2g_i(P_2) \le 0,
    \end{equation*}
    which is a contradiction. This completes the proof.
\end{proof}

\subsection*{Proof of Lemma \ref{lemma::reduction_beta}}

\begin{proof}
    By Proposition \ref{prop::feasibility2}, we have that $x_k \in \Omega$ for all $k$. Moreover, since Assumption \ref{ass::dk_feas} holds and $\alpha \in (0, 1)$, we also know that 
    \begin{equation}
        \label{eq::feas_xk}
        g_i(x_k + \alpha d_k) \le 0, \quad \forall i \in [m].
    \end{equation}
    
    We now consider the following partitions of the constraint indices $\tilde{I}_k$ and $[m] \setminus \tilde{I}_k$, one at a time.
    Note that $\Omega = \Omega_{\tilde{I}_k} \cap \Omega_{[m] \setminus \tilde{I}_k}$.
    
    \begin{enumerate}
        \item Since $k > 0$ the execution of line \ref{line::reduce_beta} in Algorithm \ref{alg::lines7-8} implies that the first condition in the if clause of line \ref{line::if-restore-2} is not satisfied: 
        \begin{equation}
            \label{eq::g_xk_sk}
            g_i(x_k + s_k) = g_i(x_k + \alpha d_k  + \beta(x_k - x_{k-1})) \le 0, \quad \forall i \in \tilde{I}_k.    
        \end{equation}
        Let us consider $h \in \mathbb{N}$ and the point $x_k + \alpha d_k  + \delta^h\beta(x_k - x_{k-1})$. Since $g$ is convex, it follows by simple algebraic manipulations that, for all $i \in \tilde{I}_k$,
        \begin{gather*}
            \begin{aligned}
                g_i(x_k + \alpha d_k  + \delta^h\beta(x_k - x_{k-1})) &= g_i((1-\delta^h)(x_k + \alpha d_k) + \delta^h(x_k + \alpha d_k + \beta(x_k - x_{k-1}))) \\&\le (1-\delta^h) g_i(x_k + \alpha d_k) + \delta^hg_i(x_k + \alpha d_k + \beta(x_k - x_{k-1})) \le 0,
            \end{aligned}
        \end{gather*}
        where the last inequality results from equations \eqref{eq::feas_xk}-\eqref{eq::g_xk_sk} and $\delta \in (0, 1)$. Since $h \in \mathbb{N}$ is arbitrary, we thus have that $x_k + \alpha d_k  + \delta^h\beta(x_k - x_{k-1}) \in \Omega_{\tilde{I}_k}$ for all $h \ge 0$.
        
        \item Let us suppose by contradiction that, for all $h \in \mathbb{N}$, there exists $i_h \in [m] \setminus \tilde{I}_k$ such that 
        \begin{equation*}
            g_{i_h}(x_k + \alpha d_k  + \delta^h\beta(x_k - x_{k-1})) > 0.
        \end{equation*}
        Since the set $[m] \setminus \tilde{I}_k$ is finite, we can consider a subsequence $H \subseteq \{0, 1, \ldots\}$ such that, for all $h \in H$, $i_h = \hat{i}$ and 
        \begin{equation*}
            g_{\hat{i}}(x_k + \alpha d_k  + \delta^h\beta(x_k - x_{k-1})) > 0.
        \end{equation*}
        Given that $g$ is continuous, we can take the limit for $h \in H$, $h \to \infty$, to obtain that $g_{\hat{i}}(x_k + \alpha d_k) \ge 0$. 
        
        Since equation \eqref{eq::feas_xk} holds, it thus follows that $g_{\hat{i}}(x_k + \alpha d_k) = 0$. Moreover, given the convexity of $g_{\hat{i}}$, $x_k \in \Omega$ and Assumption \ref{ass::dk_feas}, it is straightforward to see from the previous result that $g_{\hat{i}}(x_k + \tilde{t} d_k) = 0$, with $\tilde{t}$ defined as in Algorithm \ref{alg::CSM_constrained}. Then, it must be that $\hat{i} \in \tilde{I}_k$, which is a contradiction.
        
        Thus, there exists $\bar{h} \in \mathbb{N}$ such that, for all $i \in [m] \setminus \tilde{I}_k$, 
        \begin{equation*}
            g_i(x_k + \alpha d_k  + \delta^{\bar{h}}\beta(x_k - x_{k-1})) \le 0,
        \end{equation*}
        i.e., $x_k + \alpha d_k  + \delta^{\bar{h}}\beta(x_k - x_{k-1}) \in \Omega_{[m] \setminus \tilde{I}_k}$.
    \end{enumerate}
    Considering both cases, we conclude the proof by defining $\beta_k = \delta^{\bar{h}} \beta > 0$.
\end{proof}

\section{Algorithmic Scheme of the Extended Curve Search Method for Constrained Optimization}
\label{app::extended_scheme}

In this appendix, we provide the full algorithmic scheme of Algorithm~\ref{alg::CSM_constrained}, with all the techniques described in Section~\ref{sec::extensions} integrated. The complete scheme is presented in Algorithm~\ref{alg::extended_scheme}.

\begin{algorithm}[ht]\caption{Extended Curve Search method for smooth convexly constrained problems}
    \label{alg::extended_scheme}
	\begin{algorithmic}[1]
		\REQUIRE 
			$x_0\in\Omega$, $\alpha\in(0, 1),\ \beta>0,$ $\tilde{t} \in (0, 1)$, $\delta \in (0, 1)$, $\sigma \in (0, 1)$, $\{\varepsilon_k\} \subseteq \mathbb{R}_+$ a decreasing sequence, $M \in \mathbb{N}$, $\eta_0 > 0$, $\eta_{\max} > \eta_{\min} > 0$
		\STATE Set $k=0$
        \STATE Set $x_{-1}=x_0$
        \STATE Set $m(0) = 0$
        \WHILE{stopping criterion is not satisfied}
        \STATE Compute gradient-related direction $d_k = \Pi_{\Omega}[x_k - \eta_k\nabla f(x_k)] - x_k$
        \STATE Compute direction $s_k = \alpha d_k + \beta\eta_k(x_k - x_{k-1})$
        \STATE Compute set $\tilde{I}_k=\{i\in[m]|g_i(x_k + \tilde{t}d_k) \ge -\varepsilon_k\}$
        \IF{$\exists i\in \tilde{I}_k: g_i(x_k+s_k) > 0$ \textbf{or} $k = 0$}
        \STATE Set $s_k = d_k$ 
        \ELSE
        \IF{$d_k\ne-\eta_k\nabla f(x_k)$}
        \STATE Find $\beta_k = \max_{h \in \mathbb{N}}\{\delta^h\beta \mid x_k + \alpha d_k + \delta^h\beta\eta_k(x_k - x_{k-1}) \in \Omega\}$
        \STATE Set $s_k = \alpha d_k + \beta_k\eta_k(x_k - x_{k-1})$
        \ENDIF
        \ENDIF
        \STATE Let $\gamma_k(t)=x_k + td_k + t^2(s_k - d_k)$.
        \STATE Compute 
        \begin{equation*}
            \quad\quad
            \begin{aligned}
                t_k & = \max_{h \in \mathbb{N}}\left\{\delta^h
                \middle\vert
                \begin{aligned}
                    & \gamma_k(\delta^h) \in \Omega\ \land\\
                    &f(\gamma_k(\delta^h))\le \max_{0\leq j\leq m(k)} f(x_{k-j})+\sigma \delta^h \nabla f(x_k)^\top d_k
                \end{aligned}
                \right\}
            \end{aligned}
        \end{equation*}
        \STATE Set $x_{k+1}=\gamma_k(t_k)$.
        \STATE Set $m(k+1) = \min\{m(k) + 1, M\}$
        \STATE Set $r_k = x_{k+1} - x_k$
        \STATE Set $y_k = \nabla f(x_{k+1}) - \nabla f(x_k)$
        \STATE Set $\eta_{k+1} = \min\{\eta_{\max}, \max\{\eta_{\min}, \frac{r_k^\top r_k}{r_k^\top y_k}\}\}$
        \STATE Set $k = k + 1$
        \ENDWHILE
		\RETURN $x_k$
	\end{algorithmic}
\end{algorithm}

\end{document}